\documentclass[12pt]{mpaper}

\include{thref}
\include{symbols}

\usepackage{lmodern} 
\usepackage[T1]{fontenc} 
\usepackage[utf8]{inputenc} 

\usepackage{enumitem} 

\usepackage[hidelinks]{hyperref} 
\usepackage{doi} 

\begin{document}

\title{A jump operator on the Weihrauch degrees}

\author[Andrews]{Uri Andrews}
\author[Lempp]{Steffen Lempp}
\author[Marcone]{Alberto Marcone}
\author[Miller]{Joseph S.~Miller}
\author[Valenti]{Manlio Valenti}

\address[Uri Andrews]{Department of Mathematics\\
University of Wisconsin - Madison\\
Madison, Wisconsin 53706\\
USA}
\email{\href{mailto:andrews@math.wisc.edu}{andrews@math.wisc.edu}}

\address[Steffen Lempp]{Department of Mathematics\\
University of Wisconsin - Madison\\
Madison, Wisconsin 53706\\
USA}
\email{\href{mailto:lempp@math.wisc.edu}{lempp@math.wisc.edu}}

\address[Alberto Marcone]%
  {Dipartimento di Scienze Matematiche, Informatiche e Fisiche\\
  University of Udine\\
  33100 Udine\\
  Italy}
\email{\href{mailto:alberto.marcone@uniud.it}{alberto.marcone@uniud.it}}

\address[Joseph S.\ Miller]{Department of Mathematics\\
University of Wisconsin - Madison\\
Madison, Wisconsin 53706\\
USA}
\email{\href{mailto:jmiller@math.wisc.edu}{jmiller@math.wisc.edu}}

\address[Manlio Valenti]{Department of Mathematics\\
University of Wisconsin - Madison\\
Madison, Wisconsin 53706\\ 
USA}
\curraddr{Department of Computer Science\\
Swansea University\\
Swan\-sea, SA1 8EN\\
UK}
\email{\href{mailto:manliovalenti@gmail.com}{manliovalenti@gmail.com}}

\thanks{The authors wish to thank Damir Dzhafarov, Jun Le Goh, Takayuki Kihara, Arno Pauly, and Dilip Raghavan for useful conversations on the topics of the paper. They also thank the anonymous reviewer for their careful reading of the paper. 
The second author's research was partially supported by AMS--Simons Foundation Collaboration Grant 626304. 
The third author was partially supported by the Italian PRIN 2017 \lq\lq Mathematical Logic: models, sets, computability\rq\rq, prot.\ 2017NWTM8R and by the Italian PRIN 2022 ``Models, sets and classifications'', prot.\ 2022TECZJA, funded by the European Union - Next Generation EU.
The fourth author was partially supported by NSF Grant No.\ DMS-2053848.}

\date{}

\maketitle

\begin{abstract}
A partial order $(P,\le)$ admits a jump operator if there is a map
$j\function{P}{P}$ that is strictly increasing and weakly monotone. Despite
its name, the jump in the Weihrauch lattice fails to satisfy both of these
properties: it is not degree-theoretic and there are functions~$f$ such
that $f\weiequiv f'$. This raises the question: is there a jump operator in the
Weihrauch lattice? We answer this question positively and provide an
explicit definition for an operator on partial multi-valued functions that,
when lifted to the Weihrauch degrees, induces a jump operator. This new
operator, called the \emph{totalizing jump}, can be characterized in terms of
the total continuation, a well-known operator on computational problems.
The totalizing jump induces an injective endomorphism of the
Weihrauch degrees. We study some algebraic properties of the totalizing
jump and characterize its behavior on some pivotal problems in the
Weihrauch lattice.
\end{abstract}
\tableofcontents

\section{Introduction}

Weihrauch reducibility is a notion of reducibility between computational
problems that calibrates uniform computational strength. Despite growing interest in the Weihrauch degrees, their underlying structure
remains relatively unexplored. Early work
showed that the Weihrauch degrees form a distributive lattice with a bottom
element; see \cite{BGP17} for an overview.

In the context of classical computability theory, a central role is played by
the Turing jump. It is therefore natural to ask whether there is an analogous
operation in the Weihrauch lattice. An answer to this question requires a
precise description of the desired properties of a  jump operator.

\begin{definition}[{\cite[Def.\ 1.1]{HS91}}]\thlabel{def:jump}
Let $(P, \le)$ be a partial order. A \textdef{jump operator} on~$P$ is a
function $j\function{P}{P}$ that is
\begin{enumerate}
\item
strictly increasing, i.e., for every $p\in P$, $p< j(p)$, and
\item
(weakly) monotone, i.e., for every $p,q\in P$, if $p\le q$ then
$j(p)\le
j(q)$.
\end{enumerate}
The structure $(P, \le, j)$ is called a \textdef{jump partial order}.
\end{definition}

This definition comes from Hinman and Slaman, who showed that every countable jump
partial order is embeddable in the Turing degrees \cite[Thm.\ 1.8]{HS91}. Later, Lerman 
\cite[Thm.\ 10.1.2]{Le10} extended this result to every countable jump
partial order with least element preserved under the embedding; and Montalb\'an
\cite[Thm.\ 4.17]{Mon03} extended this result by showing that every countable
jump upper semilattice can be embedded in the Turing degrees preserving also
the join operation.

Using the Axiom of Choice, it is not hard to show that every upper
semilattice without maximum $(P, \le, \oplus)$ admits a jump operator: given
a well-ordering $\sequence{p_\alpha}{\alpha}$ of~$P$, we can define
$j(p):=p\oplus p_\alpha$, where~$\alpha$ is least such that $p_\alpha \not
\le p$. It is straightforward to check that this map is indeed a jump
operator on $(P,\le, \oplus)$. However, this argument heavily uses the Axiom
of Choice, and most likely, the defined jump operation will not be
``natural''.

In the context of Weihrauch reducibility, Brattka, Gherardi and Marcone
\cite{BolWei11} defined the \emph{jump} of a partial multi-valued function
(see Section~\ref{sec:background} for the precise definition). While this
operator (that originally was also called the \emph{derivative}) has some
connections with the Turing jump, it fails to satisfy both properties~$(1)$
and~$(2)$ in the definition of a jump operator.

In this paper, we explicitly define a jump operator on computational problems
which we call the \textdef{totalizing jump}. We show that, while the explicit
definition may look technical, it has a natural connection with the
totalization operator~$\totalization{}$, a well-known operator on
computational problems.

After recalling the necessary background notions in
Section~\ref{sec:background}, we define and study the totalizing jump in
Section~\ref{sec:tot-jump}. In particular, we show that the degree of the
totalizing jump $\jjump{f}$ of a problem~$f$ is the maximum degree of
$\totalization{g}$ for $g\weiequiv f$ (\thref{thm:jump=max_tot}). We also
show that the map~$\jjumpop$ is injective on the Weihrauch degrees
(\thref{thm:jump_endomorphism}). This, in turn, implies that~$\jjumpop$ is an
injective (but not surjective) endomorphism of the Weihrauch degrees into
themselves. As a corollary, this induces two new embeddings of the Medvedev
degrees into the Weihrauch degrees. In Section~\ref{sec:jump_of_problems}, we
explicitly characterize the totalizing jump of specific well-known problems.
We make some remarks on abstract jump operators in Section~\ref{sec:conclusions}, and finally, in Section~\ref{sec:openQ}, we highlight some open problems.

\section{Background}\label{sec:background}

In this section, we provide a short introduction to the Weihrauch degrees, focusing on what will be needed in this paper. For a more thorough presentation, the reader is referred to \cite{BGP17}.

We let~$\Baire$ and~$\Cantor$ denote Baire and Cantor space, respectively.
Let~$\baire$ and~$\cantor$ denote the sets of finite strings of natural
numbers and of finite binary sequences, respectively. We write
$\str{x_0,\hdots, x_{n-1}}$ for the string $\sigma:=i\mapsto x_i$ of
length~$n$. The length of~$\sigma$ is denoted $\length{\sigma}$. If~$x$ is a
finite or infinite string, we write $x[n]$ for the prefix of~$x$ of
length~$n$. We use $\sigma\concat \tau$ to denote the concatenation
of~$\sigma$ and~$\tau$, and~$\prefix$ for the prefix relation.

We will use the symbol $\coding{\cdot}$ to denote a fixed computable
bijection $\function*{\baire}{\mathbb{N}}$. An
explicit definition for $\coding{\cdot}$ can be found in any basic textbook
on computability theory. We assume that this map has all the standard
computability-theoretic properties, e.g., that $\sigma \mapsto
\length{\sigma}$ is computable. For the sake of readability, we write
$\coding{n_0,\hdots, n_k}$ in place of $\coding{\str{n_0,\hdots, n_k}}$.

Often, the symbol $\coding{\cdot}$ is used to denote the \textdef{join}
between two (finite or infinite) strings with the same length. The meaning of
$\coding{\cdot}$ will be clear from the context. Moreover, if
$\sequence{x_i}{i\in\mathbb{N}}$ is a sequence of infinite strings, we define
$\coding{x_0,x_1,\hdots}(\coding{i,j}) := x_i(j)$.

We write $f\pmfunction{X}{Y}$ for a partial multi-valued function with domain contained in $X$ and codomain $Y$. For every $f,g \pmfunction{\Baire}{\Baire}$, we say that~$f$ is \textdef{Weihrauch
reducible} to~$g$, and write $f\weireducible g$, if there are two computable
functionals $\Phi\pfunction{\Baire}{\Baire}$ and $\Psi\pfunction{\Baire
\times \Baire}{\Baire}$ such that, for every $p\in\dom(f)$,
\begin{enumerate}
\item $\Phi(p)\in\dom(g)$, and
\item for every $q\in g\Phi(p)$, $\Psi(p,q)\in f(p)$.
\end{enumerate}
The functionals~$\Phi$ and~$\Psi$ are often called the \textdef{forward
functional} and the \textdef{backward functional}, respectively. Unless otherwise
mentioned, we will assume that~$\Phi$ is the name for the forward functional
and~$\Psi$ is the name for the backward functional.

If~$\Psi$ need not have access to the original input~$p$, we say that~$f$ is
\textdef{strongly Weihrauch reducible} to~$g$, and write
$f\strongweireducible g$. Formally, $f\strongweireducible g$ if there are two
computable functionals~$\Phi$ and~$\Psi$ such that, for every $p\in\dom(f)$,
\begin{enumerate}
\item $\Phi(p)\in\dom(g)$, and
\item for every $q\in g\Phi(p)$, $\Psi(q)\in f(p)$.
\end{enumerate}

Weihrauch reducibility is often formulated in the more general context of
partial multi-valued functions on represented spaces, also called
computational problems. However, if we are interested in the structure of the
degrees, there is no loss of generality in assuming that computational
problems have domain and codomain~$\Baire$. Indeed, for every computational
problem on represented spaces, there is a canonical choice for a Weihrauch
equivalent problem on the Baire space (see, e.g., \cite[Lemma 3.8]{BGP17}). With a small abuse of notation, we can consider problems with other domains and codomains (e.g., $\mathbb{N}$, $\cantor$, and $\baire$). They can be identified with problems on $\Baire$ using canonical representations (e.g., $n\in\mathbb{N}$ is represented by any $p\in\Baire$ with $p(0)=n$, and a tree is represented by its characteristic function). Unless otherwise mentioned, we will only consider problems on the Baire space.

We say that~$f$ is a \textdef{cylinder} if for all~$g$, $g \weireducible f$
if and only if $g \strongweireducible f$. The notion of cylinder is often
useful to prove separation results (as proving the non-existence of a strong
Weihrauch reduction can be easier).

As mentioned in the introduction, the Weihrauch degrees form a distributive
lattice, where join~$\sqcup$ and meet~$\sqcap$ are obtained by lifting the following operators to the degrees:
\begin{itemize}
\item $(f\sqcup g)(i,x):= \{i\}\times f(x)$ if $i=0$ and
$(f\sqcup g)(i,x):= \{i\}\times g(x)$ if $i=1$;
\item $(f\sqcap g)(x,z):= \{0\} \times f(x) \cup \{1\}\times g(z)$.
\end{itemize}

There is a natural bottom element, which is the (degree of the) empty
function. The existence of a top element is equivalent to the failure of some
relatively mild form of the Axiom of Choice (see \cite[\S2.1]{BP16}). In this
paper, we work in $\mathrm{ZFC}$, so we will assume that the Weihrauch
degrees do not have a maximum element.

There is a plethora of operators defined on computational problems, each of
which captures a specific (natural) way to combine or modify computational problems.
Most (but not all) of them lift to Weihrauch degrees. It is beyond the scope
of this paper to list them all; we will instead mention the ones that are
relevant to this work.

The \textdef{parallel product} $f \times g$ is defined as $(f \times g)(x,y)
:= f(x) \times g(y)$ and captures the idea of using~$f$ and~$g$ in parallel.
Its infinite generalization is called \textdef{parallelization}, and can be
defined as the problem
$\parallelization{f}:=\sequence{x_i}{i\in\mathbb{N}}\mapsto \{
\sequence{y_i}{i\in\mathbb{N}} \st (\forall i\in\mathbb{N})(y_i \in f(x_i))
\}$. In other words, given a countable sequence of $f$-instances,
$\parallelization{f}$ produces an $f$-solution for every $f$-instance.

To capture the idea of using~$f$ and~$g$ in series, we introduce the
\textdef{compositional product}: Let $\sequence{\Gamma_p}{p\in\Baire}$ be an
effective enumeration of all partial continuous functionals with $G_\delta$
domain. We define $f\compproduct g$ as the problem that takes as input an
element of the set
\[
\{(p,x)\in \Baire \times \Baire \st x \in \dom(g) \text{ and }
   (\forall q \in g(x))\;\Gamma_p(q)\in \dom(f) \},
\]
and produces a pair $(y,w)$ with $w\in g(x)$ and $y \in f(\Gamma_p(w))$.
Historically, the compositional product is defined as a map on a pair of
computational problems or Weihrauch degrees (see \cite{BolWei11}) that
corresponds to $\max_{\weireducible} \{ f_0\circ g_0 \st f_0 \weireducible f
\text{ and } g_0 \weireducible g\}$. However, it is convenient to fix a
specific representative of such degree (see \cite{Westrick20diamond} for a
short proof of the fact that $f\compproduct g$ as defined above works). Recalling that the compositional product is associative (\cite[Prop.\ 11.5.6]{BGP17}), we
denote by $f^{[n]}$ the $n$-fold compositional product of~$f$ with itself
(i.e., $f^{[1]}:=f$, $f^{[2]}:=f\compproduct f$, and so on).

All the operators mentioned so far are degree-theoretic. We now introduce a
few operators that, despite not being degree-theoretic, still play an
important role in the theory.

The \textdef{jump}~$f'$ of $f \pmfunction{X}{Y}$ is the problem that takes as
input a convergent sequence $\sequence{p_n}{n\in\mathbb{N}}$ in~$\Baire$ and
is defined as
\[
f'(\sequence{p_n}{n\in\mathbb{N}}):= f\left(\lim_{n\to\infty}p_n\right).
\]

Observe that, letting $\mflim \pfunction{(\Baire)^\mathbb{N}}{\Baire}$ be the
computational problem that computes the limit in the Baire space,
$f'\weireducible f\compproduct \mflim$. The converse reduction does not hold
in general (take, e.g., a function~$f$ that only has computable outputs).

As anticipated, this jump operation fails to be a jump in the abstract sense:
A simple counterexample is the constant function $c:=p\mapsto 0^\mathbb{N}$
that maps every $p\in\Baire$ to the constant~$0$ string. Indeed, given that
the input plays no role, it is apparent that $c'\weiequiv c$. This shows that
the operator~$'$ is not strictly increasing. At the same, letting~$\id$ be
the identity on the Baire space, we have
\[
c' \weiequiv c \weiequiv \id \strictlyweireducible \id' \weiequiv
   \mflim,
\]
where $\id'\weiequiv\mflim$ is straightforward from the definition (see also
\cite[Ex.\ 5.3(5)]{BolWei11}).

One may think that the constant function is a somewhat weird exception, but
this is not the case. For example, as mentioned,~$f'$ intuitively
corresponds to using~$\mflim$ once, and then applying~$f$ to the result. For
any computational problem strong enough to be closed under compositional
product with~$\mflim$, the jump is not strictly increasing.

As a side note, we mention that, even though the jump is not weakly monotone on
the Weihrauch degrees, it is weakly monotone on the strong Weihrauch degrees.
It still fails to be a jump, as it is not strictly increasing on the strong
Weihrauch degrees.

In the definition of the totalizing jump, a central role is played by the
\textdef{total continuation} or \textdef{totalization} operator. For every
partial multi-valued $f\pmfunction{\Baire}{\Baire}$, its totalization is the
total multi-valued function $\totalization{f}\mfunction{\Baire}{\Baire}$
defined as
\[
\totalization{f}(x):=\begin{cases}
	f(x)	& \text{if } x\in\dom(f),\\
	\Baire	& \text{otherwise.}
\end{cases}
\]
Again, the totalization is not a degree-theoretic operation: as a simple
counterexample, it is enough to consider a total computable function and a partial computable function with
no total computable extension.

We conclude this section by listing a few computational problems that will be
useful in the rest of the paper. We already introduced the identity
problem~$\id$ and the problem
$\mflim\pfunction{{(\Baire)}^\mathbb{N}}{\Baire}$ that maps a convergent
sequence in~$\Baire$ to its limit. It is well-known that $\mflim
\strongweiequiv \parallelization{\LPO}$, where $\LPO\function{\Baire}{2}$ is
defined as $\LPO(p):=1$ iff $(\exists n\in\mathbb{N})\; p(n)\neq 0$. It is
convenient to think of $\LPO(p)$ as the answer to a single $\Sigma^{0,p}_1$- (or $\Pi^{0,p}_1$-)question.

Some benchmark examples in the Weihrauch lattice are \emph{choice problems}. The
choice problem~$\Choice{X}$ can be intuitively described as the problem of
finding elements of non-empty subsets of~$X$ given an enumeration of the complement of the subset. 
Their formal definition is usually given in the more general context of represented spaces, but for the sake of this paper, we can define them in a (strongly Weihrauch) equivalent way as problems on Baire space as follows:
\begin{itemize}
\item
$\Choice{k}$: Given $p\in (k+1)^\mathbb{N}$ such that $(\exists
n<k)\; n+1\notin\ran(p)$, find $n<k$ such that $n+1\notin\ran(p)$.
\item
$\CNatural$: Given $p\in\Baire$ such that $(\exists
n)\; n+1\notin\ran(p)$, find~$n$ such that $n+1\notin\ran(p)$.
\item
$\CCantor$: Given (the characteristic function of) an infinite subtree of~$\cantor$, find a path
through it.
\item
$\CBaire$: Given (the characteristic function of) an ill-founded subtree of~$\baire$, find a path
through it.
\end{itemize}
The restrictions of the choice problems to instances with a unique solution
are denoted with the symbol $\UChoice{X}$. It is known that
$\UChoice{\mathbb{N}}\weiequiv \CNatural$ (\cite[Thm.\ 3.8]{BolWei11}),
$\UChoice{k} \weiequiv \UChoice{\Cantor} \weiequiv \id$ (where the second
equivalence follows from the fact that~$\Cantor$ is computably compact, see,
e.g., \cite[Cor.\ 6.4]{BdBPLow12}), and  $\UCBaire \strictlyweireducible
\CBaire$ (\cite[Cor.\ 3.7]{KMP20}).

\section{The totalizing jump}\label{sec:tot-jump}

We fix a computable enumeration $\sequence{\Phi_e}{e\in\mathbb{N}}$ of
partial computable functionals from~$\Baire$ to~$\Baire$. We now introduce the following new operator on computational problems:
\begin{definition}\thlabel{def:tot-jump}
Let $f\pmfunction{\Baire}{\Baire}$ be a partial multi-valued function. We
define the \textdef{totalizing jump} (or \textdef{tot-jump} for brevity)
of~$f$ as follows: For every $e,i\in\mathbb{N}$ and every $p\in\Baire$,
\[
\jjump{f}(e,i,p):=\begin{cases} \{ \Phi_i(p, q) \st q \in f\Phi_e(p) \}
&
        \text{if }\Phi_e(p) \in \dom(f) \text{ and}\\
        &\quad (\forall q\in f\Phi_e(p))\; \Phi_i(p,q)\downarrow, \\
	\Baire	&  \text{otherwise.}
	\end{cases}
\]
This definition can be extended to problems on arbitrary represented spaces as follows. Let $(X,\delta_X)$ and $(Y, \delta_Y)$ be represented spaces and let $g\pmfunction{X}{Y}$. The \textdef{realizer version} of $g$ is the problem $g^r\pmfunction{\Baire}{\Baire}$ defined as $g^r:= \delta_Y^{-1} \circ g\circ \delta_X$. Observe that $g^r=g$ whenever $X$ and $Y$ are subsets of $\Baire$ (represented via the identity function). We define $\jjump{g}:=\jjump{g^r}$.
\end{definition}

\begin{remark}
For some proofs, it may be convenient to use the following (strongly
Weihrauch equivalent) definition for the tot-jump: For every partial
multi-valued function $f\pmfunction{\Baire}{\Baire}$ and every
$x=\str{e,i}\concat p\in \Baire$, we define
\[
\inputjjump{f}(x):=\begin{cases} \{ \Phi_i(x, q) \st q \in f\Phi_e(x) \}
&
        \text{if }\Phi_e(x) \in \dom(f) \text{ and}\\
        &\quad (\forall q\in f\Phi_e(x))\; \Phi_i(x,q)\downarrow, \\
		\Baire	&  \text{otherwise.}
		\end{cases}
\]		

To show that $\inputjjump{f} \strongweiequiv \jjump{f}$, observe that the only difference between the two problems is that in $\inputjjump{f}$, the functionals~$\Phi_e$ and~$\Phi_i$ receive as input their own indices. In particular, to prove that $\inputjjump{f}\strongweireducible \jjump{f}$, it suffices to notice  that, given $e,i,p$, we can uniformly compute $e', i'\in \mathbb{N}$ so that $\Phi_{e'}(p) = \Phi_e(\str{e,i}\concat p)$ and $\Phi_{i'}(p,q)= \Phi_i(\str{e,i}\concat p, q)$. The other reduction is proved analogously.
\end{remark}

Intuitively, we can think of the tot-jump $\jjump{f}$ of~$f$ as a problem
that ``collects all possible Weihrauch reductions to~$f$ and totalizes''. In
particular, the name ``totalizing jump'' is motivated by the following characterization
of the Weihrauch degree of~$\jjump{f}$.

\begin{theorem}\thlabel{thm:jump=max_tot}
   For every $f,g\pmfunction{\Baire}{\Baire}$, 
   \begin{itemize}
      \item if $g\weireducible f$, then $\totalization{g} \weireducible\jjump{f}$;
      \item there is $h\pmfunction{\Baire}{\Baire}$ such that $h\weiequiv f$ and $\jjump{f} \weiequiv \totalization{h}$.
   \end{itemize}
   In other words, the Weihrauch degree of $\jjump{f}$ is the maximum of the Weihrauch degrees of the totalizations of the $g$'s which are Weihrauch equivalent (equivalently, reducible) to $f$.
\end{theorem}
\begin{proof}
   Fix a problem $f$. Assume that $g\weireducible f$ via $\Phi_e,\Phi_i$. It is straightforward to see that
   $\totalization{g}\weireducible \jjump{f}$ is witnessed by the maps $p\mapsto
   (e,i,p)$ and $(p,q)\mapsto q$. Indeed, if $p\in \dom(g)$ then $\Phi_e(p)\in
   \dom(f)$ and, for every $q\in f\Phi_e(p)$, $\Phi_i(p,q)\in g(p)$. In
   particular, $\jjump{f}(e,i,p) \subseteq g(p)$. On the other hand, if
   $p\notin\dom(g)$ then $\jjump{f}(e,i,p)\subseteq \totalization{g}(p)=\Baire$.

   To conclude the proof, let us define
   \[
   h(e,i,p) := \{ \Phi_i(p,q) \st q\in f\Phi_e(p)\}
   \]
   with $\dom(h) := \{ (e,i,p) \in \mathbb{N}\times \mathbb{N}\times \Baire \st
   \Phi_e(p)\in \dom(f) \text{ and } (\forall q\in
   f\Phi_e(p))\; \Phi_i(p,q)\downarrow \}$. It is immediate from the definition
   of~$h$ that $h\weiequiv f$ and $\jjump{f}\weiequiv \totalization{h}$. 
%
%
\end{proof}

Observe that the previous theorem only deals with problems on the Baire space. In fact, if $X$ and $Y$ are arbitrary represented spaces, the total continuation of a problem $g\pmfunction{X}{Y}$ is defined as 
\[ \totalization{g}(x) := \begin{cases}
   g(x) & \text{if } x\in\dom(g)\\
   Y    & \text{otherwise.}
\end{cases}\]
In particular, choosing $Y$ so that none of its elements has a computable name, we can easily construct a counterexample to the first item of the previous theorem. 

While this could be an obstacle to lifting $\jjumpop$ to the Weihrauch degrees, we defined the tot-jump of a problem on arbitrary represented spaces as the tot-jump of its realizer version. Choosing a canonical representative for each degree allows us to focus only on problems on the Baire space. The next theorem shows that the tot-jump is a degree-theoretic operator that induces a jump operator on the Weihrauch degrees.

\begin{theorem}
For every~$f$, $f\strictlyweireducible \jjump{f}$. Moreover, for every $f,g$,
if $f\weireducible g$ then $\jjump{f}\weireducible \jjump{g}$.
\end{theorem}
\begin{proof}
   It is enough to show that the statement holds for problems on the Baire space. The reduction $f\weireducible \jjump{f}$ is straightforward (just map $x$ to $(e,i,x)$, where $e,i$ are indices for the identity function and the projection on the second component respectively), so we only need to show that
$\jjump{f}\not\weireducible f$.

Let $d\function{\Baire}{\Baire}$ be the function defined as $d(p)(0):=p(0)+1$
and $d(p)(n+1):=p(n+1)$. Observe first of all that $\inputjjump{f}$ (and
hence $\jjump{f}$) is strongly Weihrauch equivalent to the multi-valued
function~$f_d$ that, on input $x=\str{e,i}\concat p$, is defined as
\[
f_d(x):= \begin{cases} \{ d\circ \Phi_i(x, q) \st q \in f\Phi_e(x) \} &
           \text{if }\Phi_e(x) \in \dom(f) \text{ and}\\
           &\quad (\forall q\in f\Phi_e(x))\; \Phi_i(x,q)\downarrow, \\
		\Baire	&  \text{otherwise.}
	\end{cases}
\]
Indeed, the reduction $f_d\strongweireducible \inputjjump{f}$ follows
from the fact that, for every $x\in\Baire$ and every $q\in
\inputjjump{f}(x)$, $d(q)\in f_d(x)$. The converse reduction is
witnessed by the maps~$\id$ and
\[
q \mapsto \begin{cases}
		q 	& \text{if }q(0)=0,\\
		d^{-1}(q) & \text{if } q(0)>0.
	\end{cases}
\]
It is therefore enough to show that $f_d\not\weireducible f$. Assume towards
a contradiction that $f_d\weireducible f$ is witnessed by the
functionals~$\Phi_e$ and~$\Phi_i$. Fix $p\in\Baire$ and let
$y:=\str{e,i}\concat p$. Since~$f_d$ is total, $y\in\dom(f_d)$. Moreover, by
definition of Weihrauch reduction, $\Phi_e(y)\in\dom(f)$ and for every $q\in
f\Phi_e(y)$, $\Phi_i(y,q)\downarrow$.

We have now reached a contradiction, as for every non-empty $X\subseteq
\Baire$, $X\not\subseteq d(X)$ (consider $p \in X$ such that $p(0)$ is
minimal). In particular, taking $X = \{ \Phi_i(y, t) \st t \in f\Phi_e(y) \}
\neq \emptyset$, there is $q\in f\Phi_e(y)$ such that
\[
\Phi_i(y,q)\notin d(X) = \{ d\circ \Phi_i(y, t)
   \st t \in f\Phi_e(y) \} = f_d(y),
\]
contradicting the definition of Weihrauch reducibility.\footnote{Without sufficiently strong choice axioms, we cannot prove the existence of $q\in f\Phi_e(y)$ witnessing the contradiction.}

To prove the last part of the statement, assume that $f\weireducible g$ via
the functionals $\Phi_e, \Phi_i$. Let $(a,b,p)$ be an input for $\jjump{f}$.
We can uniformly compute $c,d\in\mathbb{N}$ so that $\Phi_c(p) =
\Phi_e(\Phi_a(p))$ and $\Phi_d(p,q)= \Phi_b(p,\Phi_i(\Phi_a(p),q))$. The
reduction $\jjump{f}\weireducible \jjump{g}$ is witnessed by the functionals
$(a,b,p) \mapsto (c,d,p)$ and $(p,q)\mapsto q$.

Indeed, if $\Phi_a(p)\in\dom(f)$ and, for every $q\in f\Phi_a(p)$,
$\Phi_b(p,q)\downarrow$, then $\jjump{f}(a,b,p) = \{\Phi_b(p,q) \st q\in
f\Phi_a(p)\}$. In this case, by the definition of Weihrauch reducibility,
$\Phi_c(p)=\Phi_e(\Phi_a(p))\in \dom(g)$ and for every $t\in g\Phi_c(p)$,
$\Phi_i(\Phi_a(p),t) \in f\Phi_a(p)$. In particular, $\Phi_b(p,
\Phi_i(\Phi_a(p),t))\downarrow \in \jjump{f}(a,b,p)$. In other words,
$\jjump{g}(c,d,p) \subseteq \jjump{f}(a,b,p)$. The other case (i.e., if
$\Phi_a(p)\notin\dom(f)$ or if there is $q\in f\Phi_a(p)$ such that
$\Phi_b(p,q)\uparrow$) is trivial as $\jjump{g}(c,d,p)\subseteq
\jjump{f}(a,b,p)=\Baire$.
\end{proof}

\begin{remark}
Notice that the same proof shows that~$\jjumpop$ induces a jump operator on
the strong Weihrauch degrees. Indeed, the reductions
$f\weireducible\jjump{f}$  and $\jjump{f}\weireducible \jjump{g}$ (when
$f\weireducible g$) are both strong Weihrauch reductions. Moreover, as
noticed, $f_d \strongweiequiv \inputjjump{f} \strongweiequiv \jjump{f}$,
hence $f\strictlystrongweireducible \jjump{f}$.
\end{remark}

Observe that the definition of $\jjump{f}$ is $\Delta^{1,f}_2$ in the
language of third-order arithmetic, i.e., there is a $\Delta^1_2$-formula
with parameter~$f$ that says ``$t\in \jjump{f}(e,i,p)$''. Indeed,
\begin{itemize}
\item
$\alpha_f(p) := \Phi_e(p)\downarrow \in\dom(f)$ is equivalent to $(\exists r,q)[\; r = \Phi_e(p) \land (r,q)\in f \;]$, which is $\Sigma^{1,f}_1$;
\item
$\beta_f(p):= (\forall q)[\; (\Phi_e(p),q)\in f \rightarrow
\Phi_i(p,q)\downarrow \;]$ is $\Pi^{1,f}_1$;
\end{itemize}
hence the formula $t \in \jjump{f}(e,i,p)$ can be written as
\[
(\alpha_f(p) \land \beta_f(p))\rightarrow (\exists q \in f\Phi_e(p))\;  t = \Phi_i(p,q).
\]

\begin{remark}\thlabel{rem:no_sigma11_def}
No jump operator on partial multi-valued functions can be defined by a
$\Sigma^{1,f}_1$ formula. To show this, we use the fact that $\CBaire
\weiequiv \codedChoice{\boldfaceSigma^1_1}{}{\Baire}$ (essentially proved in
\cite[Thm.\ 3.11]{KMP20}), where $\codedChoice{\boldfaceSigma^1_1}{}{\Baire}$ is the problem of finding elements in non-empty analytic subsets of $\Baire$. Assume that~$\mathfrak{j}$ is an operator defined
by a $\Sigma^{1,f}_1$-formula. 
In particular, $q \in \mathfrak{j}(\CBaire)(p)$ is a $\Sigma^1_1$-formula $\varphi(p,q)$ where~$\CBaire$ may appear. 
Since ``$x\in \CBaire(A)$'' is arithmetic in $x,A$ (in fact, $\CBaire(A)$ is~$\Pi^0_1$ relative to~$A$), any arithmetic formula involving it is arithmetic as well. This implies that $\varphi(p,q)$ is actually $\Sigma^1_1$ uniformly in~$p,q$ and hence we can use $\codedChoice{\boldfaceSigma^1_1}{}{\Baire}$ to pick a point in $\{ q \in \Baire \st \varphi(p,q)\}$. 
In other words, $\mathfrak{j}(\CBaire) \weireducible \codedChoice{\boldfaceSigma^1_1}{}{\Baire} \weiequiv \CBaire$, so $\mathfrak{j}$ is not strictly increasing.
\end{remark}

It is straightforward to see that the map~$\jjumpop$ is injective. More
interestingly, it is injective on the Weihrauch degrees.

\begin{theorem}\thlabel{thm:jump_endomorphism}
For every $f,g$, if $\jjump{f}\weireducible\jjump{g}$ then $f\weireducible
g$. This implies that the map~$\jjumpop$ is an injective endomorphism on the
Weihrauch degrees.
\end{theorem}
\begin{proof}
Let $f,g$ be two partial multi-valued functions and assume that
$\jjump{f}\weireducible \jjump{g}$ via the functionals~$\Phi$
and~$\Psi$. Consider a pair $\str{e,i}$ such that~$e$ is an index for~$\id$
and $\Phi_i(p,q)$ is defined as follows: Let $m\in \mathbb{N}$ be the first number found such
that $(\exists z\in\Baire)\; \Psi((e,i,p), z)(0)\downarrow = m$.
Then
\[
\Phi_i(p,q)(n) := \begin{cases}
		q(0)+m+1 	& \text{if } n=0,\\
		q(n)				& \text{otherwise.}
	\end{cases}
\]

To show that $f\weireducible g$, let $p\in\dom(f)$ and consider the input
$(e,i,p)$ for $\jjump{f}$. Let $(a,b,t) = \Phi(e,i,p)$ be an input for
$\jjump{g}$. Observe that $\Phi_a(t)\in \dom(g)$ and for every $r\in
g\Phi_a(t)$, $\Phi_b(t,r)\downarrow$. Indeed, if not, then any $z\in \Baire$
is a valid solution for $\jjump{g}(a,b,t)$. In particular, we could take~$z$
so that $\Psi((e,i,p), z)(0)\downarrow = m$. This would lead to a
contradiction as, by definition of~$\Phi_i$, for every $y\in
\jjump{f}(e,i,p)$ we have $y(0)>m$. In other words,
\[
\jjump{g}(\Phi(e,i,p)) = \{ \Phi_b(t, r) \st (a,b,t) = \Phi(e,i,p)
\text{ and } r\in g\Phi_a(t) \}.
\]
Hence, a solution for $\jjump{f}(e,i,p)$, and in turn for $f(p)$, can be
uniformly obtained from any element of $g(\Phi_a(t))$.
\end{proof}

As we will discuss extensively later,~$\jjumpop$ is not surjective, even on the cone above $\jjumpop(\emptyset)$. The
previous theorem implies that there is a proper substructure of the Weihrauch
degrees that is isomorphic to the Weihrauch degrees. Note that, using this
endomorphism, we obtain two (and, by iterating, infinitely many) new
embeddings of the Medvedev degrees into the Weihrauch degrees (see
\cite{Hinman2012} for a survey on Medvedev reducibility, and see \cite[Thm.\
9.1]{BGP17} for the two known embeddings of the Medvedev degrees into the
Weihrauch degrees).

Observe that, as a corollary of \thref{thm:jump=max_tot}, we obtain that $\jjump{f}$ is never a cylinder.
Indeed, we can prove something slightly stronger:

\begin{proposition}
For every $f\pmfunction{\Baire}{\Baire}$ such that $f
\strictlystrongweireducible\totalization{f}$, $\totalization{f}$ is not a
cylinder.
\end{proposition}

\begin{proof}
It is well-known that~$g$ is a cylinder iff $\id \times g \strongweireducible
g$ (\cite[Cor.\ 3.6]{BG11}). Assume towards a contradiction that $\id \times
\totalization{f} \strongweireducible \totalization{f}$ is witnessed by the
functionals $\Phi, \Psi$. Notice that, for some computable $p\in\Baire$ and
some $x\in\Baire$, $\Phi(p,x) = z$ for some $z\notin\dom(f)$. Indeed, if this
were not the case, then we would obtain $\totalization{f}\strongweireducible
f$, contradicting the hypothesis.

Since $\totalization{f}(z)=\Baire$, for every $q\in\Baire$ we obtain
$\Psi(q)=(p,y)$ for some $y \in \totalization{f}(x)$. If we consider $(p',x)$
with $p'\neq p$ and $q \in \totalization{f}\Phi(p',x)$, we reach a
contradiction, as $\Psi(q) = (p,y) \notin (\id \times \totalization{f})
(p',x)$.
\end{proof}

Since, as proved in \thref{thm:jump=max_tot}, for every~$f$ there is
$g\weiequiv f$ such that $g\strictlyweireducible\jjump{g}\weiequiv \jjump{f}=\totalization{g}$, the
previous proposition implies that $\jjump{f}$ is not a cylinder.

In the rest of the section, we prove several properties of the tot-jump,
including various results that better describe the range of $\jjumpop$. We first provide an alternative characterization of $\jjump{f}$. For this, we
introduce the following computational problem:

\begin{definition}
Let us define $\W\pmfunction{\Baire}{\Baire}$ as
\[
\W(p):= \{ q \in \Baire \st (\forall i)\; q(i+1)>q(i) \text{ and }
   p\circ q = 0^\mathbb{N} \}.
\]
In other words, $\W(p)$ lists the addresses of infinitely many zeroes of~$p$.
\end{definition}

Notice that~$\W$ is uniformly computable (it can be solved by unbounded search) and
partial, as $\dom(\W)=\{ p\in\Baire \st (\exists^\infty
i)\; p(i)=0\}$. An important property of~$\W$ is that, for any given
$p\in\dom(\W)$ and any $q\in \Baire$, it is c.e.\ to check if $q\notin
\W(p)$. This also motivates the choice of the notation, as~$\W$ is
``translating a $\Pi^0_2$-question into a $\Pi^0_1$-question''.\footnote{It is probably more correct to say that $\totalization{(\W)}$ translates a $\Pi^0_2$-question into a $\Pi^0_1$-question. We discuss this computational problem in Section~\ref{sec:jump_of_problems}.}

\begin{theorem}\thlabel{thm:jump=w*f*w}
For every $f$, $\jjump{f}\weiequiv \totalization{(\W \compproduct f
\compproduct \W)}$.
\end{theorem}

\begin{proof}
Since~$\W$ is uniformly computable, $\W\compproduct f \compproduct \W
\weireducible f$, so, in light of \thref{thm:jump=max_tot}, we only need to show
that $\jjump{f}\weireducible \totalization{(\W \compproduct f \compproduct
\W)}$. For every input $(e,i,p)$ for $\jjump{f}$, let $t\in \Baire$ be such
that~$t$ has infinitely many zeroes iff $\Phi_e(p)\downarrow$. Let also $v\in
\Baire$ be such that $\Gamma_v(q)=\Phi_e(p)$ for every~$q$ and $w\in\Baire$
be such that $\Gamma_w(q_0,q_1)$ has infinitely many zeroes iff
$\Phi_i(p,q_0)\downarrow$. It is clear that $t,v,w$ are uniformly computable
from $e,i,p$. The forward functional~$\Phi$ of the reduction
$\jjump{f}\weireducible \totalization{(\W \compproduct f \compproduct \W)}$
is the map $(e,i,p) \mapsto (w,v,t)$. We define the backward
functional~$\Psi$ as follows: A solution for $\totalization{(\W \compproduct
f \compproduct \W)}$ is a string of the form $\coding{y_1,y_2,y_3}$. Given
$(e,i,p)$ and $\coding{y_1,y_2,y_3}$, we compute $\Phi(e,i,p)= (w,v,t)$ and
do the following operations in parallel:
\begin{itemize}
\item check whether $y_3\in \W(t)$;
\item check whether $y_1\in \W(\Gamma_w(y_2,y_3))$;
\item compute $\Phi_i(p,y_2)$.
\end{itemize}
Since it is c.e.\ to check if $y_3\notin \W(t)$ or $y_1\notin
\W(\Gamma_w(y_2,y_3))$, as long as~$y_1$ and~$y_3$ ``appear to be correct'', the
backward functional produces $\Phi_i(p,y_2)$. If we see that $y_3\notin
\W(t)$ or $y_1\notin \W(\Gamma_w(y_2,y_3))$, we extend the partial output with
$0^\mathbb{N}$. 	

Recall that, if $\Phi_e(p)\in \dom(f)$ and, for every $q\in f\Phi_e(p)$,
$\Phi_i(p,q)\downarrow$, then $\jjump{f}(e,i,p)= \{ \Phi_i(p,q) \st q\in
f\Phi_e(p) \}$. It is straightforward to check that, in this case, $t\in
\dom(\W)$, $\Gamma_v(y)=\Phi_e(p)\in \dom(f)$ for every~$y$, and for every
$q\in f\Phi_e(p)$, $\Gamma_w(q, \Gamma_v(y))$ has infinitely many zeroes (as $\Phi_i(p,q)\downarrow$). In particular,
every solution $\coding{y_1,y_2,y_3}$ of $\totalization{(\W \compproduct f
\compproduct \W)}(\Phi(e,i,p))$ is such that $y_1\in \W(\Gamma_w(y_2,y_3))$,
$y_2\in f\Phi_e(p)$, and $y_3\in \W(t)$. By definition, the backward
functional will therefore compute $\Phi_i(p,y_2) \in \jjump{f}(e,i,p)$.

On the other hand, if $\Phi_e(p)\notin \dom(f)$ or if there is $q\in
f\Phi_e(p)$ such that $\Phi_i(p,q)\uparrow$, then $\jjump{f}(e,i,p) =
\Baire$. Since $\totalization{(\W \compproduct f \compproduct \W)}$ and
$\Psi$ are total, the claim follows.
\end{proof}

The following proposition shows that, in general, the compositions with $\W$
on both sides are necessary.

\begin{proposition}
There is~$f$ such that
$\jjump{f}\not\weireducible\totalization{(\W\compproduct f)}$ and
$\jjump{f}\not\weireducible \totalization{(f\compproduct\W)}$.
\end{proposition}

\begin{proof}
Let $A\in \Cantor$ be such that $\emptyset \strictlyturingreducible A
\strictlyturingreducible \emptyset'$. Let~$f$ be the function with
$\dom(f):=\{A\}$ that maps~$A$ to $\emptyset'$.

We first show that $\jjump{f}\not\weireducible \totalization{(\W\compproduct
f)}$. Assume towards a contradiction that the reduction is witnessed by the
maps $\Phi, \Psi$. Let~$i$ be an index for the projection on the second coordinate. Let also~$e$ be such that $\Phi_e(x)$ searches for
$k\in\mathbb{N}$ such that $x = 0^k1\concat p$ for some $p\in\Baire$ and then
outputs~$p$. Since $(e,i,0^{\mathbb{N}})\in \dom(\jjump{f})$ and~$A$ is not
computable, there is~$k$ such that $\Phi(e,i,0^k) = \pairing{\sigma,\tau}$
and $\tau \neq A[|\tau|]$. 
In particular, for every $p\in\Baire$, $y :=
\Phi(e,i,0^k\concat p)\notin \dom(\W \compproduct f)$, and therefore
$0^{\mathbb{N}}\in \totalization{(\W\compproduct f)}(y)$. On the other hand,
it is immediate to check that $\jjump{f}(e,i,0^k1\concat A) = \emptyset'$.
We have reached a contradiction, as $\Psi((e,i,0^k1\concat A), 0^\mathbb{N}) = \emptyset'$
would imply that $\emptyset'\turingreducible A$, against the hypothesis
on~$A$.

Let us now show that $\jjump{f}\not\weireducible
\totalization{(f\compproduct\W)}$. To this end, assume towards a
contradiction that the reduction is witnessed by the functionals~$\Phi$
and~$\Psi$. The idea is to diagonalize by choosing an input $(e,i,x)$ for
$\jjump{f}$ so that the output of~$\Phi_i$ is different from any output
of~$\Psi$. Let~$e$ be an index for the identity functional and let $x=A$.

To find $i\in\mathbb{N}$ we use the recursion theorem. First we define
$\Phi_i(p,q)(0)$: In parallel, we compute $\Psi((e,i,p),\coding{\sigma_1,\sigma_2})$ for all possible $\sigma_1,\sigma_2\in\baire$
until we find a pair $(\sigma_1,\sigma_2)$ such that
\[
\Psi((e,i,p),\coding{\sigma_1,\sigma_2})(0)\downarrow = m,
\]
for some $m\in\mathbb{N}$ and then set $\Phi_i(p,q)(0):=m+1$. At least one
such pair exists, otherwise~$\Psi$ is never defined when the first input is
$(e,i,p)$, contradicting the definition of Weihrauch reducibility.

Since we are interested in defining the behavior of $\Phi_i(p,q)$ only when
$p=x=A$ and $q=\emptyset'$, we describe a procedure for computing
$\Phi_i(x,\emptyset')(1)$ which on different inputs may not converge, or
converge to an arbitrary string. Start from $\pairing{w,z} = \Phi(e,i,x)$ and
use~$\emptyset'$ to check if there is a $\tau\in\baire$ that satisfies all
the following properties:
\begin{itemize}
\item
$(\forall j<\length{\tau}-1)(\tau(j+1)>\tau(j) )$,
\item
$(\forall j<\length{\tau})(z\circ \tau(j) = 0)$,
\item
$\Psi((e,i,x), \pairing{\emptyset'[\length{\tau}],
\tau})(1)\downarrow$.
\end{itemize}
Intuitively, we are searching for an initial segment~$\tau$ of a solution of
$\W(z)$. A solution for $f\compproduct \W(w,z)$ is of the form $(r,s)$, where
$s\in \W(z)$ and $r\in f(\Gamma_w(s))$. Since~$f$ is constant~$\emptyset'$,
to obtain a prefix of a solution we only need to search for a prefix~$\tau$
of~$s$. To diagonalize, we require that~$\tau$ is sufficiently long so that
$\Psi((e,i,x), \pairing{\emptyset'[\length{\tau}], \tau})(1)\downarrow$. Such a~$\tau$ need not exist, as
we do not know if~$z$ has infinitely many zeroes. This is the reason why we
use the oracle to check if such a search terminates.

If~$\tau$ exists, then we can search for it and define
$\Phi_i(x,\emptyset')(1):=\Psi((e,i,x),
\pairing{\emptyset'[\length{\tau}], \tau})(1)+1$. Otherwise, define
$\Phi_i(x,\emptyset')(1):=0$.

For every $n>1$ and every~$p$ and~$q$, define $\Phi_i(p,q)(n):=0$.

Notice that $\jjump{f}(e,i,A) = \Phi_i(A,\emptyset')$. If the input~$z$
for~$\W$ has finitely many~$0$'s or, for some $s\in \W(z)$, $\Gamma_w(s)\neq
A$, then $\pairing{\sigma_1,\sigma_2}$ is the initial segment of a solution
for $\totalization(f\compproduct\W)(w,z)$, hence $\Phi_i(A,\emptyset')(0)\neq
\Psi((e,i,A),\coding{\sigma_1,\sigma_2}\concat 0^\mathbb{N})(0)$.
Otherwise, $\pairing{\emptyset'[\length{\tau}], \tau}$ is the prefix of a
solution for $f\compproduct\W(w,z)$, and therefore $\Phi_i(A,\emptyset')(1)\neq
\Psi((e,i,A),\pairing{\emptyset'[\length{\tau}],
\tau})(1)$.
\end{proof}

The previous result shows that, in general, the use of~$\W$ cannot be avoided
on either side of~$f$. There are, however, many~$f$ such that
$\jjump{f}\weiequiv \totalization{f}$. We now provide a sufficient condition
for this to happen.

\begin{theorem}\thlabel{thm:totjump=totalization}
Fix a problem~$f$. If there are two total computable functions
$\varphi,\psi\function{\mathbb{N}\times\mathbb{N}}{\mathbb{N}}$ such that 
\begin{itemize}
\item
for every $e,i$, $\varphi(e,i)$ and $\psi(e,i)$ are indices of total
functionals and
\item
whenever $g\weireducible f$ via $\Phi_e, \Phi_i$ (which might be
partial), then $g\weireducible f$ via $\Phi_{\varphi(e,i)}$ and
$\Phi_{\psi(e,i)}$,
\end{itemize}
then $\jjump{f}\weiequiv \totalization{f}$.
\end{theorem}

\begin{proof}
By \thref{thm:jump=max_tot}, we only need to show that
$\jjump{f}\weireducible \totalization{f}$. We let the forward functional of
the reduction be defined by $\Phi(e,i,p) := \Phi_{\varphi(e,i)}(p)$.
Similarly, we let the backward functional be defined by $\Psi((e,i,p),q) :=
\Phi_{\psi(e,i)}(p,q)$. The proof is then straightforward: Notice indeed that
if $(e,i,p)$ is an input for $\jjump{f}$ such that $\Phi_e(p)\in\dom(f)$ and
for every $q\in f\Phi_e(p)$, $\Phi_i(p,q)\downarrow$, then the functionals
$\Phi_e,\Phi_i$ are witnessing the reduction $g\weireducible f$ for some
problem~$g$ (e.g., we can take~$g$ to be the problem that maps~$p$ to the set
$\{\Phi_i(p,q) \st q \in f\Phi_e(p) \}$). In particular, the second item in
the hypotheses implies that  $\Phi(e,i,p)=\Phi_{\varphi(e,i)}(p)\in \dom(f)$
and, for every $q\in f\Phi(e,i,p)$, $\Psi((e,i,p),q) = \Phi_{\psi(e,i)}(p,q) \in
g(p) = \jjump{f}(e,i,p)$.

On the other hand, if $\Phi_e(p)\notin\dom(f)$ or if there is $q\in
f\Phi_e(p)$ such that $\Phi_i(p,q)\uparrow$, then $\jjump{f}(e,i,p)=\Baire$,
hence the claim follows by the totality of $\Phi,\Psi$ (guaranteed by the
first item in the hypotheses).
\end{proof}

Intuitively, the hypotheses of the previous result require that any reduction
$g\weireducible f$ witnessed by $\Phi_e, \Phi_i$ is, in fact, witnessed by
total functionals, and the indices for such functionals can be found
uniformly in $e,i$. \thref{thm:totjump=totalization} can be rephrased as
follows:

\begin{corollary}\thlabel{thm:jump=tot}
Fix a partial multi-valued function~$f$. Assume there are total computable
functionals~$\Phi$ and~$\Psi$ such that for every $e,i\in\mathbb{N}$ and
every $p\in\Baire$, if $\Phi_e(p)\in\dom(f)$ and for every $q\in f\Phi_e(p)$,
$\Phi_i(p,q)\downarrow$, then we have $\Phi(e,i,p)\in\dom(f)$ and
$\{ \Psi((e,i,p),q) \st q \in f\Phi(e,i,p)\} \subseteq \{ \Phi_i(p,q) \st q \in f\Phi_e(p)\}$.

Then $\jjump{f}\weiequiv \totalization{f}$.
\end{corollary}

This rewording, on top of showing a closer connection with the definition of
the tot-jump, allows us to draw a connection with the notion of
\textdef{transparent} functions, introduced in \cite{BolWei11}. More
precisely, a function $H\pfunction{\Baire}{\Baire}$ is called
\textdef{transparent} if for every computable $F\pfunction{\Baire}{\Baire}$
there is a computable $G\pfunction{\Baire}{\Baire}$ such that $F\circ H = H
\circ G$.

This notion can be naturally generalized to computational problems by
requiring that given two (indices for) functionals $\Phi_e,\Phi_i$ witnessing
$g\weireducible f$, we can uniformly compute the index of a function
$\Phi\pfunction{\Baire}{\Baire}$ such that the reduction $g\weireducible f$ is
witnessed by the functionals~$\Phi$ and $(p,q)\mapsto q$.

We can strengthen this ``generalized transparency'' by requiring that $\Phi$ is a total functional. This strengthening implies that $\jjump{f}\weiequiv \totalization{f}$.
Indeed, these conditions are equivalent to requiring that there is a total
computable~$\Phi$ such that, for every $e,i\in\mathbb{N}$ and every
$p\in\Baire$, if $\Phi_e(p)\in\dom(f)$ and for every $q\in f\Phi_e(p)$,
$\Phi_i(p,q)\downarrow$, then
\[
f\Phi(e,i,p) \neq \emptyset \quad \text{and}\quad  f\Phi(e,i,p)
   \subseteq
   \{ \Phi_i(p,q) \st q \in f\Phi_e(p)\}.
\]

While many problems (including~$\mflim$ and its iterations) satisfy the above
conditions, the assumption that the backward functional is exactly the
projection on the second component is unnecessarily strong. In fact, any
total computable functional would serve the same purpose. Thus we have the
hypotheses of \thref{thm:jump=tot}. 

In Section~\ref{sec:jump_of_problems}, we will use \thref{thm:jump=tot} to
describe the tot-jump of some natural problems.

\begin{lemma}\thlabel{thm:fractal}
There are $\bar e,\bar \imath\in\mathbb{N}$ such that for every~$f$ and every
$k\in\mathbb{N}$, $\jjump{f}$ is Weihrauch equivalent to its restriction to
$X_k:= \{ \str{\bar e,\bar \imath}\concat 0^k \concat x \st x\in\Baire \}$.
\end{lemma}

\begin{proof}
We let $\bar e, \bar \imath$ be the indices of two universal functionals such
that $\Phi_{\bar e}(p)= \Phi_n(p)$ and $\Phi_{\bar \imath}(p,q)=\Phi_m(p,q)$
where $\coding{n,m}$ is least such that $p(\coding{n,m})=1$. It is obvious
that any restriction of $\jjump{f}$ is Weihrauch reducible to $\jjump{f}$. To
prove that $\jjump{f}\weireducible\jjump{f}\restrict{X_k}$, we let the
forward functional be the map $\Phi(e,i,p):=\str{\bar e,\bar
\imath}\concat 0^{\coding{n,m}}1\concat p$, where $n,m> k$ are such that, for
every~$t$,
\begin{align*}
		\Phi_n(0^t1\concat p) &= \Phi_e(p), \text{ and } \\
		\Phi_m(0^t1\concat p, q) &= \Phi_i(p,q).
\end{align*}
Clearly such $n,m$ can be uniformly computed from $e,i,k$. The backward
functional is the projection on the second coordinate. Note that $\Phi_{\bar e}(0^{\coding{n,m}}1\concat p)=
\Phi_{n}(0^{\coding{n,m}}1\concat p) = \Phi_e(p)$, and similarly $\Phi_{\bar
\imath}(0^{\coding{n,m}}1\concat p, q) = \Phi_m(0^{\coding{n,m}}1\concat p, q) =
\Phi_i(p,q)$. Thus $\jjump{f}(e,i,p)= \jjump{f}(\bar e, \bar
\imath, 0^{\coding{n,m}}1\concat p)$, which concludes the proof.
\end{proof}

\begin{theorem}\thlabel{thm:total_and_join_irreducible}
For every~$f$, $\jjump{f}$ is total and join-irreducible.
\end{theorem}

\begin{proof}
The fact that $\jjump{f}$ is total is apparent by definition. We show that if
$\jjump{f}\weiequiv g_0 \sqcup g_1$ then there is $b<2$ such that
$\jjump{f}\weireducible g_b$, and thus $g_{1-b} \weireducible g_b$. Assume
that the reduction $\jjump{f}\weireducible g_0 \sqcup g_1$ is witnessed by
the functionals~$\Phi$ and~$\Psi$ and let~$\bar e$ and~$ \bar \imath$ be as
in \thref{thm:fractal}. By the continuity of the forward functional, there
exist~$k$ and $b<2$ such that $\Phi(\bar e,\bar \imath, 0^k)(0)\downarrow =
b$. In particular, for every $x\extends 0^k$, the functional~$\Phi$ produces
an input for~$g_b$. Since, by \thref{thm:fractal}, $\jjump{f}$ is equivalent
to its restriction to $\{ \str{\bar e,\bar \imath}\concat 0^k \concat x \st
x\in\Baire \}$, the claim follows.
\end{proof}

We will show in \thref{thm:j_not_onto_total_join-irred} that the range
of~$\jjumpop$ is a proper subset of the total, join-irreducible degrees.

\begin{corollary}
For every $f,g$, $\jjump{f}\sqcup \jjump{g} \weireducible \jjump{f \sqcup
g}$. Moreover, the reduction is strict iff $f \weiincomparable g$.
\end{corollary}

\begin{proof}
The fact that $\jjump{f}\sqcup \jjump{g} \weireducible \jjump{f \sqcup g}$
follows by the monotonicity of~$\jjumpop$ and the fact that~$\sqcup$ is the
join in the Weihrauch degrees. Clearly, if $f\weireducible g$  then $f \sqcup
g \weiequiv g$, hence $\jjump{f} \sqcup \jjump{g} \weiequiv\jjump{g}
\weiequiv \jjump{f\sqcup g}$. Conversely, if $f\weiincomparable g$ then
$\jjump{f} \weiincomparable \jjump{g}$ by \thref{thm:jump_endomorphism}, and
hence $\jjump{ f \sqcup g} \not \weiequiv \jjump{f}\sqcup \jjump{g}$ (by
\thref{thm:total_and_join_irreducible}), which implies that $\jjump{f}\sqcup
\jjump{g} \strictlyweireducible \jjump{f \sqcup g}$.
\end{proof}

\begin{theorem}
For every~$f$, there is~$h$ such that $f\strictlyweireducible h
\strictlyweireducible \jjump{f}$.
\end{theorem}

\begin{proof}
We distinguish three cases: If $\id \weireducible f$, then the claim
immediately follows from the fact that the Weihrauch degrees are dense
above~$\id$ (see \cite{LMPSV23}).

If $\id \weiincomparable f$, then $f\strictlyweireducible f \sqcup \id =: h$.
Moreover, $\id\weireducible \jjump{f}$ (as $\jjump{f}$ is total), hence
$f\sqcup \id \weireducible \jjump{f}$. The fact that the reduction is strict
follows from the fact that $\jjump{f}$ is join-irreducible
(\thref{thm:total_and_join_irreducible}).

If $\emptyset \strictlyweireducible f \strictlyweireducible \id$, then we have $f \strictlyweireducible \id \strictlyweireducible \jjump{f}$ by $\jjump{\emptyset} \weiequiv \id$ (Theorem \ref{thm:J(0)=id} below) and the injectivity of the tot-jump proved in Theorem \ref{thm:jump_endomorphism}.
Finally, if $f=\emptyset$ then $\jjump{f} \weiequiv \id$ by Theorem \ref{thm:J(0)=id}, and the existence of $h$ with $\emptyset \strictlyweireducible h \strictlyweireducible \id$ is well-known.
\end{proof}

\begin{theorem}
The range of~$\jjumpop$ is not dense, i.e., there exist~$f$ and~$g$ such that
$f\strictlyweireducible g$, the set $X:=\{ h \st
\jjump{f}\strictlyweireducible h \strictlyweireducible \jjump{g}\}$ is
non-empty and, for every $h\in X$, $h\notin \ran(\jjumpop)$.
\end{theorem}

\begin{proof}
The claim follows immediately from the fact that the Weihrauch degrees are dense above~$\id$ and that there are (strong) minimal
covers in the Weihrauch lattice (see \cite{LMPSV23}). Indeed, it is enough to choose~$f$ and~$g$ such
that~$g$ is a minimal cover of~$f$. Since the tot-jump is always total,
$\ran(\jjumpop)$ is contained in the cone above~$\id$ and there exists~$h$
such that $\jjump{f}\strictlyweireducible h \strictlyweireducible \jjump{g}$.
By \thref{thm:jump_endomorphism}, if $h\weiequiv \jjump{h_0}$ then
$f\strictlyweireducible h_0 \strictlyweireducible g$, contradicting the fact
that~$g$ is a minimal cover of~$f$.
\end{proof}

\begin{proposition}\thlabel{thm:tj-meet}
   For every $f,g$, $\jjump{f \sqcap
   g} \weireducible \jjump{f}\sqcap \jjump{g}$. There exist~$f$ and~$g$ such that $\jjump{f\sqcap g} \strictlyweireducible
\jjump{f}\sqcap \jjump{g}$.
\end{proposition}

\begin{proof}
   The first part of the statement is straightforward using the fact that $\jjumpop$ is monotone and that $\sqcap$ is the meet in the Weihrauch lattice. 
   
   To prove the second part of the statement, let $X,Y\subseteq \Baire$ be two incomparable Turing degrees (i.e., for some Turing-incomparable $x,y\subseteq \mathbb{N}$, $X$ and $Y$ are, respectively, the equivalence classes of $x$ and $y$ under $\turingequiv$), and let $f:=\id\restrict{X}$ and $g:=\id\restrict{Y}$. Assume towards a contradiction that the reduction $\jjump{f}\sqcap \jjump{g} \weireducible \jjump{f\sqcap g}$ is witnessed by the functionals~$\Phi$ and~$\Psi$. We claim that one of $\jjump{f}$ or $\jjump{g}$ is uniformly computable. This contradicts \thref{thm:J(0)=id} below, which implies that $\jjump{\emptyset}$ is the only uniformly computable tot-jump.
   
   Let $\bar e, \bar \imath$ be as in \thref{thm:fractal}. Observe that, for
   every~$k$ and every $x\in X$,
   \[
   0^\mathbb{N}\in \jjump{f\sqcap g}(\Phi( (\bar e, \bar \imath, 0^k
   \concat x),
   (\bar e, \bar \imath, 0^k\concat x) )).
   \]
   Indeed, if we let $(e, i, z)$ be the input for $\jjump{f \sqcap g}$ given by
   the value $\Phi( (\bar e, \bar \imath, 0^k \concat x), (\bar e, \bar \imath,
   0^k\concat x) )$, then $\Phi_e(z)\notin\dom(f \sqcap g)$ as, by
   hypothesis,~$X$ and~$Y$ are Turing-incomparable (in particular~$x$ does not
   compute any $y\in Y$). Analogously, by swapping the roles of $X$ and $Y$ in the above argument, for every $k\in\mathbb{N}$ and $y\in Y$,
   $0^\mathbb{N}\in \jjump{f\sqcap g}(\Phi( (\bar e, \bar \imath, 0^k \concat
   y), (\bar e, \bar \imath, 0^k\concat y) ))$.
   
   By continuity, there is $k\in\mathbb{N}$ such that $\Psi( ((\bar e, \bar
   \imath, 0^k),(\bar e, \bar \imath, 0^k)), 0^\mathbb{N})(0)$ commits to some
   $b<2$. Without loss of generality, we can assume that $b=0$, i.e., the
   backward functional commits to producing a solution for $\jjump{f}$. We
   therefore obtain that
   \[
   \Psi( ((\bar e, \bar \imath, 0^k\concat x),
   (\bar e, \bar \imath, 0^k\concat x)), 0^\mathbb{N}) \in \jjump{f}
   (\bar e, \bar \imath, 0^k\concat x ),
   \]
   i.e., we can uniformly compute a solution for $\jjump{f}(\bar e,\bar \imath,
   0^k\concat x)$ which, in turn, implies that $\jjump{f}$ is uniformly
   computable.
\end{proof}

The previous proof can be generalized by letting $X,Y$ be
two incomparable Muchnik degrees (see \cite{Hinman2012} for the definition of
Muchnik reducibility).

\smallskip

We conclude this section by providing some (relatively) weak sufficient conditions for not being in the range of $\jjumpop$. 

\begin{definition}
We call a partial multi-valued function~$f$ \textdef{co-total} if, for every
problem~$g\pmfunction{X}{Y}$, where $X$ and $Y$ are represented spaces,
\[ f\weireducible \totalization{g} \iff f\weireducible g.\tag{$\star$}\]
A problem $f$ is called \textdef{Baire co-total} if $(\star)$ holds for all problems $g\pmfunction{\Baire}{\Baire}$.
\end{definition}

The notion of co-totality was introduced in \cite[Def.\ 4.13]{BGCompOfChoice19}. Clearly, Baire co-totality is a weaker property, as $g$ ranges over a smaller family of problems. Intuitively, a problem $f$ is co-total when the totalization operator ``cannot help'' to solve it. To show that Baire co-totality is strictly weaker than co-totality, consider the following simple counterexample: let $\mathrm{NC}\subset \Baire$ be the set of non-computable points, and define $f(0^\mathbb{N}):=\mathrm{NC}$. Let also $X:=\{x_0,x_1\}$ and $Y:=\{y\}$ be two represented spaces, where $x_0$ has a computable name while $x_1$ and $y$ only have non-computable names. Define $g\pfunction{X}{Y}$ as $g(x_1):=y$. Now observe that $f$ is not co-total, as $f\not\weireducible g$, but $f\weireducible \totalization{g}$. On the other hand, $f$ is Baire co-total: indeed, for every $h\pmfunction{\Baire}{\Baire}$, if $f\weireducible \totalization{h}$ via $\Phi, \Psi$, then the two functionals already witness $f\weireducible h$. This is because if $\Phi(0^\mathbb{N})\notin \dom(h)$ then $0^\mathbb{N}\in\totalization{h}(\Phi(0^\mathbb{N}))=\Baire$, but $\Psi(0^\mathbb{N},0^\mathbb{N})\notin f(0^\mathbb{N})$.

We now observe that $f$ is Baire co-total exactly when the tot-jump does not give any useful information to solve $f$. 

\begin{theorem}\thlabel{thm:co-total_char}
   For every~$f$,~$f$ is Baire co-total iff for every~$g$,
   \[ f\weireducible \jjump{g} \iff f\weireducible g. \]
   In particular, if $f$ is (Baire) co-total, then $f\notin\ran(\jjumpop)$.
\end{theorem}

\begin{proof}
   The proof is straightforward using \thref{thm:jump=max_tot}. Indeed, assume
   that~$f$ is Baire co-total and let~$g$ be such that $f \weireducible \jjump{g}$. Without loss of generality, we can assume that $g$ is a problem on the Baire space: if not, we can replace $g$ with its realizer version $g^r$ as, by definition, $\jjump{g}=\jjump{g^r}$.
   Since $\jjump{g}\weiequiv \totalization{g_0}$ for some problem $g_0$ on the Baire space with $g_0\weiequiv g$, we
   obtain $f\weireducible \totalization{g_0}$ and hence $f\weireducible g_0
   \weiequiv g$.
   
   The converse implication is similar: Assume that for every~$g$,
   $f\weireducible \jjump{g}$ implies $f \weireducible g$, and let~$h\pmfunction{\Baire}{\Baire}$ be such
   that $f\weireducible \totalization{h}$. Since $\totalization{h}\weireducible
   \jjump{h}$, we immediately obtain $f\weireducible h$.

   Finally, if $f$ is Baire co-total and $f\weiequiv \jjump{h}$ for some $h$, then $f\weireducible h$, against $\jjump{h}\not\weireducible h$.
\end{proof}
   
In \cite{BGCompOfChoice19}, several problems are proved to be co-total including~$\CNatural$ and~$\CCantor$, hence we immediately have the following result:

\begin{corollary}\thlabel{thm:CN_notin_ran(tj)}
The problems~$\CNatural$ and~$\CCantor$ are not in $\ran(\jjumpop)$. \qed
\end{corollary}

The next theorem leads to a sufficient condition for a function to be Baire co-total. Let $\mathsf{U}\pfunction{\Baire}{\Baire}$ be a fixed universal Turing
functional. Let $\DIS\mfunction{\Baire}{\Baire}$ be the problem defined as
\[
\DIS(p):=\{ q \in\Baire \st \mathsf{U}(p)\neq q\}.
\]
It is immediate from the definition that~$\DIS$ is total. In fact, for every
$p\in\Baire$, $\DIS(p)$ is either~$\Baire$ (if $\mathsf{U}(p)\uparrow$) or
$\Baire\setminus \{ \mathsf{U}(p)\}$. The problem~$\DIS$ was studied
extensively in \cite{BrattkaDis} and is one of the weakest discontinuous
problems.\footnote{In \cite{BrattkaDis}, it was shown that, under
$\mathrm{ZF+DC+AD}$,~$\DIS$ is a strong minimal cover of~$\id$ in the
topological version of Weihrauch reducibility. Such a result cannot be
transferred to the (plain) Weihrauch degrees, as the cone above~$\id$ is
dense in the Weihrauch degrees under~$\mathrm{ZFC}$~\cite{LMPSV23}.} In
particular, $\DIS \weireducible \Choice{2}$ (see, e.g., \cite[Prop.\
5.10]{BrattkaStashing}).

\begin{theorem}[with Arno Pauly]\thlabel{thm:dis_x_g}
If $\DIS \times g \weireducible \jjump{f}$, then $g \weireducible f$.
\end{theorem}

\begin{proof}
Let~$\Phi$ and~$\Psi$ witness the reduction $\DIS \times g \weireducible
\jjump{f}$. Let $\mathsf{U}\pfunction{\Baire}{\Baire}$ be the universal computable
functional used in the definition of $\DIS$. By the recursion theorem, there
is a computable $p\in \Baire$ such that $\mathsf{U}(\pairing{p,x})$ is the first
component of $\Psi((\pairing{p,x},x),0^\mathbb{N})$. Consider the reduction
$g \weireducible \jjump{f}$ where the forward functional is given by
$x\mapsto \Phi(\pairing{p,x},x)$ and the backward functional maps $(x,y)$ to
the second component of $\Psi((\pairing{p,x},x),y)$.

We claim that if $x\in\dom(g)$, then $\jjump{f}(\Phi(\pairing{p,x},x))\neq
\Baire$ and so it does not fall in the ``otherwise'' case. This implies that
our reduction of~$g$ to $\jjump{f}$ is essentially a Weihrauch reduction of~$g$
to~$f$. To prove the claim, note that if $x\in\dom(g)$, then
$(\pairing{p,x},x)\in\dom(\DIS\times g)$. If $0^\mathbb{N}\in
\jjump{f}(\Phi(\pairing{p,x},x))$, then
$\Psi((\pairing{p,x},x),0^\mathbb{N})$ must converge to an element $(q,z)\in
\DIS(\pairing{p,x})\times g(x)$. But $\DIS(\pairing{p,x})$ must be different
from~$q$, the first component of  $\Psi((\pairing{p,x},x),0^\mathbb{N})$.
\end{proof}

\begin{corollary}\thlabel{thm:product_dis}
If $\DIS \times g \weireducible g$, then~$g$ is Baire co-total and hence it is not
in the range of~$\jjumpop$. \qed
\end{corollary}

As mentioned,~$\DIS$ is quite weak and hence being closed under parallel
product with~$\DIS$ is a rather weak condition satisfied by many natural
problems, like~$\CNatural$, $\CCantor$,~$\mflim$, $\UCBaire$, and~$\CBaire$.

\begin{proposition}
$\DIS$ is not Baire co-total.
\end{proposition}

\begin{proof}
Observe that $\DIS = \totalization{(\DIS \restrict{X})}$, where $X:= \{
p\in\Baire \st \mathsf{U}(p)\downarrow\}$. To show that~$\DIS$ is not
Baire co-total, it is enough to show that $\DIS \not \weireducible
\DIS\restrict{X}$. This follows from the fact that
$\DIS\restrict{X}\weireducible\id$: indeed, for every~$p$ such that
$\mathsf{U}(p)\downarrow$, it is enough to consider $q:=n \mapsto
\mathsf{U}(p)(n)+1$.
\end{proof}

Combining this result with \thref{thm:product_dis}, we immediately obtain:

\begin{corollary}
$\DIS\times\DIS \not\weireducible \DIS$. \qed
\end{corollary}

\begin{proposition}
   \thlabel{thm:dis_join_irreducible}
   $\DIS$ is join-irreducible.
\end{proposition}
\begin{proof}
   Without loss of generality, we can assume that $\mathsf{U}(\str{e}\concat p)=\Phi_{e}(p)$. Let $e$ be such that $\Phi_e(0^j 1\concat p)= \Phi_j(p)$ and $\Phi_e(0^\mathbb{N})\uparrow$. Notice that $\DIS \weireducible \DIS \restrict{X}$, where $X:=\{ x\in \Baire \st x(0)=e \}$ (it is enough to map $p$ to $\str{e}\concat 0^i1\concat p$, where $i$ is an index of $\mathsf{U}$).
   
   We show that if $\DIS\restrict{X}\weireducible f_0\sqcup f_1$, then $\DIS\restrict{X}\weireducible f_0$ or  $\DIS\restrict{X}\weireducible f_1$. If $\Phi,\Psi$ witness the reduction $\DIS\restrict{X}\weireducible f_0\sqcup f_1$ (in particular, for every input $q$ for $\DIS$, $\Phi(q)$ produces a pair $(b,x)$ with $x\in\dom(f_b)$), then by continuity, there is $k$ such that $\Phi(\str{e}\concat 0^k)(0)=b<2$. Observe that we can uniformly map any $x\in X$ of the form $\str{e}\concat 0^j1\concat p$ to $\str{e}\concat 0^h 1\concat p$ for some $h>k$ such that $\Phi_j(p)=\Phi_h(p)$. 
   The reduction $\DIS\restrict{X}\weireducible f_b$ is therefore witnessed by the maps $\str{e}\concat 0^j1\concat p\mapsto \Phi(\str{e}\concat 0^h 1\concat p)(1)$ and $\Psi$.
\end{proof}

While being join-irreducible and not being Baire co-total are necessary conditions for a Weihrauch degree to be in the range of 
$\jjumpop$, we will show in \thref{cor:DIS_notin_ran(tJ)} that they are not sufficient, as $\DIS$ is not equivalent to the tot-jump of any problem. In other words, the range of
the tot-jump is a proper subset of the set of total, non Baire co-total, join-irreducible degrees.

\section{The jump of specific problems}\label{sec:jump_of_problems}

In this section, we explicitly characterize the tot-jump of various
well-known problems. Let us start with a straightforward example.

\begin{theorem}\thlabel{thm:J(0)=id}
$\jjump{\emptyset} \weiequiv \id$.
\end{theorem}

\begin{proof}
The proof is trivial as for every $e,i,p$, $\jjump{\emptyset}(e,i,p)=\Baire$,
as there are no $e,p$ such that $\Phi_e(p)\in \dom(\emptyset)$. Therefore,
$\jjump{\emptyset}$ is total and uniformly computable, and hence equivalent
to~$\id$.
\end{proof}

To characterize~$\jjump{\id}$, we first introduce the following problem.

\begin{definition}
Let us define  $\XPi\mfunction{\Baire}{\Baire}$ as
\[
\XPi(p) := \{ q \in \Baire \st (\exists^\infty n)(p(n)=0) \iff
   (\forall n)(q(n)=0) \}.
\]
\end{definition}

Intuitively,~$\XPi$ transforms a $\Pi^0_2$-question into a
$\Pi^0_1$-question.
An alternate form of $\XPi$ was introduced by Neumann and Pauly~\cite{NP18}, who defined
\[
\mathsf{isFinite}_\Sier(p) := \{ q \in \Baire \st (\exists^\infty n)(p(n)=1)
    \iff (\forall n)(q(n)=0) \}.
\]
It is immediate that $\XPi \weiequiv \mathsf{isFinite}_\Sier$.

\begin{proposition}\thlabel{thm:Xpi=totW}
$\XPi \weiequiv \totalization{(\W)}$.
\end{proposition}

\begin{proof}
For the left-to-right reduction, recall that, given $p,q\in\Baire$, it is
c.e.\ to check whether $q\notin \W(p)$. In particular, given~$p$ and $q\in
\totalization{(\W)}(p)$, we can uniformly compute $t\in\Baire$ defined as
$t(n):=0$ if $q(n+1)>q(n)$ and $p\circ q(n)=0$, and $t(n):=1$ otherwise. It
is apparent that $t\in \XPi(p)$.

Similarly, for the right-to-left reduction, given $p\in\Baire$ and $q\in
\XPi(p)$, we can compute a solution for $\totalization{(\W)}(p)$ as follows:
Let $\sigma_0:=\str{}$. For every~$n$, if $q(i)=0$ for all $i<n$ we check if
$p(n)=0$. If it is, we define $\sigma_{n+1}:=\sigma_n \concat \str{n}$,
otherwise we let $\sigma_{n+1}:=\sigma_n$. If instead $q(i)>0$ for some
$i<n$, we let $\sigma_{n+1} := \sigma_n \concat \str{0}$. It is
straightforward to check that $r:=\bigcup_n \sigma_n \in \Baire$ is uniformly
computable from~$p$ and~$q$ and that $r\in \totalization{(\W)}(p)$.
\end{proof}

\begin{theorem}\thlabel{thm:j(id)=x_pi}
$\jjump{\id} \weiequiv \XPi$.
\end{theorem}

\begin{proof}
For the left-to-right reduction, observe that~$\jjump{\id}$ can be written as
follows:
\[
\jjump{\id}(e,i,x) = \begin{cases}
		\Phi_i( x, \Phi_e(x)) &
             \text{if } \Phi_e(x)\downarrow \land\, \Phi_i( x, \Phi_e(x))\downarrow, \\
		\Baire 	& \text{otherwise.}
	\end{cases}
\]
Since the domain of a computable functional is a $\lightfacePi^0_2$-set, we
can uniformly compute $p\in\Baire$ such that
\[
(\exists^\infty n) \, p(n)=0 \iff  \Phi_e(x)\downarrow \land\, \Phi_i( x, \Phi_e(x))\downarrow.
\]
Given $q\in \XPi(p)$, we are able to uniformly compute a solution for
$\jjump{\id}(e,i,x)$ as follows: In parallel, run $\Phi_i( x,
\Phi_e(x))$ and check whether there is~$n$ such that $q(n)\neq 0$. If no
such~$n$ is found then we are producing $\Phi_i( x,
\Phi_e(x))$, which is the correct solution for $\jjump{\id}(e,i,x)$. Otherwise,
it means that $\jjump{\id}(e,i,x)= \Baire$, hence we can stop simulating
$\Phi_i( x, \Phi_e(x))$ and continue the output
with~$0^\mathbb{N}$.

For the right-to-left reduction, note that $\W\weireducible \id$, so \thref{thm:jump=max_tot} and \thref{thm:Xpi=totW} give us $\XPi\weiequiv \totalization{(\W)} \weireducible \jjump{\id}$.
\end{proof}

With a more careful analysis, we can characterize the $n$-th tot-jump
of~$\id$. Intuitively, we can think of $\jjumpn{\id}{n}$ as a problem capturing the
following: You are allowed to ask~$n$ many $\Sigma^0_2$-questions serially.
For every $j< n$, you can see in finite time if the answer to the $j$-th
question is ``yes'' and then you can ask the next question. However, if the
$j$-th answer is ``no'', then the procedure hangs and it is impossible to see
the answers of the remaining questions.

\begin{theorem}
   For every $p\in\Baire$, let $A^p_n := \{ k < n \st (\exists^\infty j)\; p(j)=k \} $. 
For every $n>0$, let $g_n\mfunction{\Baire}{\Baire}$ be defined as
\[
g_n(p) := \begin{cases}
		\left\{0^{\sigma(0)}1^{\sigma(1)}\hdots
            (m-1)^{\sigma(m-1)}m^\mathbb{N }\st
            \sigma \in \mathbb{N}^m \right\} &
            \text{if } A^p_n \neq \emptyset \text{ and }m := \min A^p_n,\\
		\left\{0^{\sigma(0)}1^{\sigma(1)}\hdots
            (n-1)^{\sigma(n-1)}n^\mathbb{N }\st
            \sigma \in \mathbb{N}^n \right\} &
            \text{if } A^p_n = \emptyset.
	\end{cases}
\]
Then $\jjumpn{\id}{n}\weiequiv g_n$.
\end{theorem}

\begin{proof}
By induction on~$n$: The base case $n=1$ is \thref{thm:j(id)=x_pi}, as it is straightforward to see that $g_1\weiequiv \XPi$. For the induction step, it suffices to show that $\jjump{g_n} \weiequiv
g_{n+1}$.  	

Let us first prove that $\jjump{g_n}\weireducible g_{n+1}$. Let $(e,i,x)$ be
an input for $\jjump{g_n}$. For every $k < n$, we can uniformly compute~$y_k$
so that $\ran(y_k) \subseteq \{k, n+1\}$ and~$k$ occurs infinitely many times
in~$y_k$ if and only if
\[
\Phi_e(x){\downarrow} \land k \in A^{\Phi_e(x)}_n \land{}
   (\forall \sigma \in \mathbb{N}^k)\; \Phi_i(x,0^{\sigma(0)}1^{\sigma(1)}
   \hdots (k-1)^{\sigma(k-1)}k^\mathbb{N})\downarrow.
\]
This can be done because the displayed formula is uniformly~$\Pi^0_2$ in
$(e,i,x)$. Similarly, we can uniformly compute~$y_n$ so that $\ran(y_n)
\subseteq \{n, n+1\}$ and~$n$ occurs infinitely many times in~$y_n$ if and
only if
\[ \Phi_e(x){\downarrow} \land (\exists^\infty j)(\Phi_e(x)(j)\ge n)\land (\forall \sigma \in \mathbb{N}^n)\; \Phi_i(x,0^{\sigma(0)}1^{\sigma(1)}\hdots (n-1)^{\sigma(n-1)}n^\mathbb{N})\downarrow. \]
Let $y:=\pairing{y_0,\hdots, y_n}$. We now claim that a solution for
$\jjump{g_n}(e,i,x)$ can be uniformly computed from any $z \in g_{n+1}(y)$ as
follows: We compute $\Phi_i(x, z)$ as long as $z(j)\le n$. If $z(j)=n+1$ for some~$j$, we stop the computation and continue the output
with~$0^\mathbb{N}$.

Let us now show that this procedure correctly produces a solution for
$\jjump{g_n}(e,i,x)$.
\begin{itemize}[wide, itemsep=4pt, topsep=6pt]
\item
If $\Phi_e(x)\uparrow$, then for every $k\le n$,~$y_k$ only has finitely
many~$k$'s. This implies that $(\forall^\infty j)\; y(j)=n+1$, hence $n+1\in
\ran(z)$. The procedure produces an eventually null string, which is
clearly a valid solution as $\jjump{g_n}(e,i,x)=\Baire$.
\item
If $\Phi_e(x)\downarrow$ and $A^{\Phi_e(x)}_n = \emptyset$ then for every
$k<n$ and for almost all $j$, $y_k(j)=n+1$; moreover,~$n$ occurs infinitely
many times in~$y_n$ if and only if, for all $\sigma \in \mathbb{N}^n$,
$\Phi_i(x,0^{\sigma(0)}1^{\sigma(1)}\hdots
(n-1)^{\sigma(n-1)}n^\mathbb{N})\downarrow$. This, in turn, implies that if
$z \in g_{n+1}(y)$ and $n+1\notin\ran(z)$ then $\Phi_i(x,z)\downarrow \in
\jjump{g_n}(e,i,x)$, while otherwise $\jjump{g_n}(e,i,x)=\Baire$ and the
procedure computes an eventually null string.

\item
Finally, assume that $\Phi_e(x)\downarrow$ and $A^{\Phi_e(x)}_n\neq
\emptyset$. Let $k:=\min A^{\Phi_e(x)}_n$ and notice that for every $k'<k$
we have $(\forall^\infty j)\; y_{k'}(j)=n+1$.
Moreover, $y_k$ has infinitely
many~$k$ if and only if, for all $\sigma \in \mathbb{N}^k$, $\Phi_i(x,0^{\sigma(0)}1^{\sigma(1)}\hdots
(k-1)^{\sigma(k-1)}k^\mathbb{N})\downarrow$. If~$y_k$ has infinitely
many~$k$, then we can just run the computation $\Phi_i(x,z)$ for any $z \in
g_{n+1}(y)$. Otherwise, $\jjump{g_n}(e,i,x)=\Baire$ and the described
procedure is guaranteed to produce an infinite string (and therefore a
valid solution).
\end{itemize}	

Let us now show $g_{n+1}\weireducible \jjump{g_n}$. Intuitively, the
reduction works as follows: The forward functional maps an input~$p$ for
$g_{n+1}$ to $(e,i,p)$, where~$e$ is an index for the identity functional
and~$i$ is an index for the functional that, given $(p,q)$, tries to produce
a list of positions witnessing the fact that $q\in g_n(p)$. More precisely,
if~$q$ is of the form $0^{\sigma(0)}1^{\sigma(1)}\hdots
(k-1)^{\sigma(k-1)}k^\mathbb{N}$ for some $\sigma\in \baire$ and some $k\le
n$, then $\Phi_i(p,q)$ produces a strictly increasing string such that for
some strictly increasing sequence $\sequence{v_j}{j\in\mathbb{N}}$ and for
every~$j$,
\[
p(\Phi_i(p,q)(j)) = q(v_j).
\]
This can be done iteratively as follows: At stage~$0$, search for $u_0, v_0$
such that $p(u_0)=q(v_0)$. At stage $j+1$, we search for $u_{j+1}>u_j$ and
$v_{j+1}>v_j$ such that $p(u_{j+1})=q(v_{j+1})$. The sequence
$\sequence{u_j}{j\in\mathbb{N}}$ is the output of $\Phi_i(p,q)$.

If instead~$q$ is not of the form $0^{\sigma(0)}1^{\sigma(1)}\hdots
(k-1)^{\sigma(k-1)}k^\mathbb{N}$ for any $\sigma\in \baire$ and any $k\le n$,
we let $\Phi_i(p,q)\uparrow$.

Given $z\in \jjump{g_n}(e,i,p)$, the backward functional~$\Psi$ is defined as
$\Psi(p,z)(0) := \min\{ p(z(0)), n\}$ and $\Psi(p,z)(j+1):= p( z(j+1))$ if
$p( z(j)) \le p( z(j+1)) \le n$ and $z(j+1)>z(j)$, and $\Psi(p,z)(j+1):= n+1$
otherwise.

To conclude the proof we show that, for every $z\in \jjump{g_n}(e,i,p)$,
$\Psi(p,z)$ correctly produces a solution for $g_{n+1}(p)$. Observe that if
$A^p_{n+1}\neq \emptyset$,
then $g_{n+1}(p) =g_n(p)$. In particular, for every~$j$,
\[
\Psi(p,z)(j) = p(z(j)) = p(\Phi_i(p,q)(j)) = q(v_j)
\]
for some $q\in g_n(p)$ and some strictly increasing
$\sequence{v_j}{j\in\mathbb{N}}$, and hence $\Psi(p,z)$ is a correct solution for
$g_{n+1}(p)$. On the other hand, if $A^p_{n+1}= \emptyset$ then, for every $t\in g_n(p)$, $\Phi_i(p,t)\uparrow$. This implies that
$\Psi(p,z)$ is of the form $0^{\sigma(0)}1^{\sigma(1)}\hdots
n^{\sigma(n)}(n+1)^\mathbb{N}$ for some~$\sigma$, and therefore is a
valid solution for $g_{n+1}(p)$.
\end{proof}

We now show that the set $\{ \jjump{f} \st f \weireducible \id \}$ is a
proper subset of $\{ h \st \jjump{\emptyset} \weireducible h \weireducible
\jjump{\id}\}$.

\begin{theorem}\thlabel{thm:id<f<tJ(id)}
Let $f \strictlyweireducible \jjump{\id}$ and let $X:=\{ p\in\dom(f) \st p $ is computable$\}$. If\/ $\id \strictlyweireducible f\restrict{X}$, then there is no~$h$ such that $f\weiequiv \jjump{h}$.
\end{theorem}
\begin{proof}
By \thref{thm:jump_endomorphism}, if $\jjump{h}\weiequiv f
\strictlyweireducible \jjump{\id} $ then $h\strictlyweireducible \id$, i.e.,
$h\weiequiv \id\restrict{A}$ for some  $A\subseteq \Baire$ (as the lower cone
of~$\id$ is isomorphic to the dual of the Medvedev degrees, see, e.g.,
\cite[Sec.\ 5]{HiguchiPauly13}). Notice that~$A$ (and therefore $\dom(h)$) does not have any computable point, as otherwise $h \weiequiv \id$. 	

Assume $f\restrict{X} \weireducible \jjump{\id\restrict{A}}$ via $\Phi, \Psi$. Let $p\in X$
be a computable input for~$f$ and let $\Phi(p)= (e,i,x)$. Observe
that, since~$A$ has no computable point, $\Phi_e(x)\notin A$, hence
$\jjump{\id\restrict{A}}(e,i,x)=\Baire$. This implies that a reduction $f\restrict{X}
\weireducible \jjump{\id\restrict{A}}$ would yield a reduction of $f\restrict{X}$ to the
(uniformly computable) map $p\mapsto \Baire$, contradicting
$f\restrict{X}\not\weireducible\id$.
\end{proof}

\begin{theorem}\thlabel{thm:LPO<=J(id)}
$\LPO\strictlyweireducible\jjump{\id} \strictlyweireducible
\parallelization{\jjump{\id}} \weiequiv \mflim$.
\end{theorem}

\begin{proof}

   Let us first show that $\LPO\strictlyweireducible\jjump{\id}$. For the reduction, let $p \in \Baire$ be an input for~$\LPO$. By \thref{thm:j(id)=x_pi}, we can
use~$\jjump{\id}$ to compute $q\in \Baire$ such that $(\forall n)\; q(n)=0$ if and only if 
$(\exists n)\; p(n)>0$. It is then straightforward to see that we can find
an answer for $\LPO(p)$ by unbounded search (either there is a non-zero
element in~$q$ or a non-zero element in~$p$). To prove that the reduction is strict, recall that $\jjump{\id}\weiequiv \XPi\weiequiv \mathsf{isFinite}_\Sier$ (\thref{thm:j(id)=x_pi}). The reduction $\jjump{\id}\weireducible \LPO$ would contradict $\mathsf{isFinite}_\Sier\not\weireducible\TCN$ (\cite[Prop.\ 24(5)]{NP18}).

The second reduction is immediate and the reduction $\mflim \weireducible
\parallelization{\jjump{\id}}$ follows from the fact that
$\parallelization{\LPO} \weiequiv \mflim$. 

Given that~$\mflim$ is parallelizable, to prove that $\parallelization{\jjump{\id}} \weireducible \mflim$ it suffices to show that $\jjump{\id} \weireducible \mflim$.
To this end, we prove that $\XPi\weireducible \mflim$, and the claim will
follow from \thref{thm:j(id)=x_pi}. Given $p\in \Baire$, we can uniformly
compute the converging sequence $\sequence{q_n}{n\in\mathbb{N}}$ defined as
$q_n(m):=0$ if there is~$j$ such that $m \leq j \leq n$ such that $p(j)=0$,
and $q_n(m):=1$ otherwise. Clearly, for each~$m$, $\lim_{n\to\infty} q_n(m)=
0$ if and only if there is some~$0$ in~$p$ after position~$m$.
Therefore, $\lim_{n\to\infty} q_n \in \XPi(p)$.

Finally the fact that $\jjump{\id} \strictlyweireducible
\parallelization{\jjump{\id}}$ follows from the fact that every computable input for $\jjump{\id} \weiequiv \totalization{(\W)}$ has a computable solution, while this is not the case for $\mflim$.
\end{proof}

Given that $\DIS\weireducible \LPO$, combining \thref{thm:id<f<tJ(id)} and \thref{thm:LPO<=J(id)} we obtain:

\begin{corollary}\thlabel{cor:DIS_notin_ran(tJ)}
   The problems~$\DIS$ and~$\LPO$ are not Weihrauch-equivalent to any problem in the range of~$\jjumpop$. \qed
\end{corollary}

Since we showed that $\DIS$ is total, non Baire co-total, and join-irreducible, we also obtain the following corollary:

\begin{corollary}\thlabel{thm:j_not_onto_total_join-irred}
   The map $\jjumpop$ does not induce a surjective operator onto the total, non Baire co-total, join-irreducible degrees. \qed
\end{corollary}

Next we show that no lower cone is closed under tot-jump, i.e., there are no tot-jump principal ideals in the Weihrauch lattice.

\begin{theorem}
For every $g\neq\emptyset$, there exists an $f \strictlyweireducible g$ such
that $\jjump{f}\not\weireducible g$.
\end{theorem}

\begin{proof}
If $\id \not\weireducible g$, then $f := \emptyset$ has the desired
properties, while if $g \weiequiv \id$ we can set $f := \id \restrict{X}$ for
any~$X$ without computable elements (in this case, $\emptyset
\strictlyweireducible f \strictlyweireducible \id$ and we can use
\thref{thm:jump_endomorphism}).	We can now assume that $\id \strictlyweireducible
g$ and distinguish two cases.

The first one is when there exists~$g_0$ with finite domain such that $g
\weiequiv g_0$. In this case, we claim that $\LPO\not\weireducible g$. Granting
the claim, $f := \id$ has the desired property by \thref{thm:LPO<=J(id)}. To
prove the claim, assume that~$\Phi$ and~$\Psi$ witness $\LPO \weireducible
g_0$. Then, since every point in $\dom(g_0)$ is isolated, there are $z\in
\dom(g_0)$ and~$k_0$ such that for every $x \extends 0^{k_0}$, $\Phi(x) = z$.
Besides, there is $k_1 \geq k_0$ such that for some $y\in g_0(z)$,
$\Psi(0^{k_1},y)(0)\downarrow = 0$. In particular, the string
$0^{k_1}1^\mathbb{N}$ witnesses the fact that~$\Phi$ and~$\Psi$ do not
witness the Weihrauch reduction.

Assume now that~$g$ is not Weihrauch equivalent to any problem with finite
domain. We want to define~$f$ such that $f\strictlyweireducible g$ and
$\jjump{f}\not\weireducible g$. To this end, we define a
\emph{scrambling function} $\xi\pfunction{\dom(g)}{\mathbb{N}}$. The desired
$f$ will be defined as $f(x,\xi(x)):= g(x)$, with $\dom(f):=\{
(x,n)\in\dom(g)\times \mathbb{N}\st \xi(x)=n \}$. Notice that, no matter
which~$\xi$ we choose, $f\weireducible g$.

To define~$\xi$ we define a sequence $\sequence{\xi_s}{s\in\mathbb{N}}$ of
functions with finite domain, starting with $\xi_0:=\emptyset$.

At stage $s+1=2\coding{e,i}+1$, we satisfy the requirement
``$g\not\weireducible f$ via $\Phi_e,\Phi_i$''. To do so, we choose (in a
noneffective way) some $x\in \dom(g)$ such that one of the following
conditions holds:
\begin{itemize}
\item
$\Phi_e(x)\uparrow$;
\item
$\Phi_e(x)$ produces the pair $(y,n)$ with $y\notin \dom(g)$ or
$\xi_s(y)\downarrow \neq n$;
\item
$\Phi_e(x)$ produces the pair $(y,n)$ and there is $q \in g(y)$ such
that $\Phi_i(x,q)\notin g(x)$;
\item
$\Phi_e(x)$ produces the pair $(y,n)$ with $\xi_s(y)\uparrow$.
\end{itemize}
Such an~$x$ must exist, as otherwise~$\Phi_e$ and~$\Phi_i$ witness $g
\weiequiv f_s$, where $f_s(x,\xi(x)):=g(x)$ with $\dom(f_s):= \{(x,n) \st
\xi_s(x) \downarrow = n \}$, contradicting the fact that~$g$ is not Weihrauch
equivalent to any problem with finite domain. In the first three cases, there
is nothing to do, and we just define $\xi_{s+1}:=\xi_s$. In the last case, we
let $\xi_{s+1}:=\xi_s \cup \{ (y,n+1)\}$.

At stage $s+1=2\coding{e,i}+2$, we satisfy the requirement
``$\jjump{f}\not\weireducible g$ via $\Phi_e,\Phi_i$''. Let $\tilde{\xi}$ be
an arbitrary computable extension of~$\xi_s$ and let
$\tilde{f}:=(x,\tilde{\xi}(x)) \mapsto g(x)$. Clearly $\tilde{f}\weiequiv g$,
hence $\jjump{\tilde{f}}\weiequiv \jjump{g} \not \weireducible g$. In
particular, there are $\tilde{e},\tilde{\imath},\tilde{x}$ witnessing
$\jjump{\tilde{f}}\not \weireducible g$ via $\Phi_e, \Phi_i$.

If $\Phi_e(\tilde{e},\tilde{\imath},\tilde{x})\notin\dom(g)$ or if there is $q\in g \Phi_e(\tilde{e},\tilde{\imath},\tilde{x})$ such that $\Phi_i((\tilde{e},\tilde{\imath},\tilde{x}),q)\uparrow$ then there is nothing to do and we can define $\xi_{s+1}:=\xi_s$. Otherwise, the fact that $(\tilde{e},\tilde{\imath},\tilde{x})$ witnesses $\jjump{\tilde{f}}\not\weireducible g$ means that there is $q\in g\Phi_e(\tilde{e},\tilde{\imath},\tilde{x})$ such that $\Phi_i((\tilde{e},\tilde{\imath},\tilde{x}),q)\downarrow\notin\jjump{\tilde{f}}(\tilde{e},\tilde{\imath},\tilde{x})$. In particular, this implies that $\jjump{\tilde{f}}(\tilde{e},\tilde{\imath},\tilde{x})\neq \Baire$, i.e.\ 
\begin{itemize}
   \item $\Phi_{\tilde{e}}(\tilde{x})\downarrow =: (\tilde{z},\tilde{n})$
   \item for all $r\in \tilde{f}(\tilde{z},\tilde{n})$, $\Phi_{\tilde{\imath}}(\tilde{x},r)\downarrow$
   \item $\Phi_i((\tilde{e},\tilde{\imath},\tilde{x}),q)\notin \jjump{\tilde{f}}(\tilde{e},\tilde{\imath},\tilde{x}) = \{\Phi_{\tilde{\imath}}(\tilde{x},r) \st r\in \tilde{f}(\tilde{z},\tilde{n}) \}$.
\end{itemize}
We define $\xi_{s+1}:=\xi_s \cup \{ (\tilde{z},\tilde{n}) \}$.  Observe that, if $\tilde{z} \in\dom(\xi_s)$ then $\xi_s(\tilde{z})= \tilde{n}$, so $\xi_{s+1}$. This is enough to satisfy the current requirement: indeed, let $\zeta$ be any extension of~$\xi_{s+1}$ and let $h:=(x,\zeta(x)) \mapsto g(x)$. In particular, $h(\tilde{z},\tilde{n})=\tilde{f}(\tilde{z},\tilde{n})$, hence the triple $(\tilde{e},\tilde{\imath},\tilde{x})$ witnesses $\jjump{h}\not\weireducible g$ via $\Phi_e,\Phi_i$. 

The desired scrambling function is the map $\xi := \bigcup_{s\in\mathbb{N}}
\xi_s$. Since all the requirements are satisfied, the above-defined
function~$f$ satisfies $f\strictlyweireducible g$ and $\jjump{f}\not
\weireducible g$, which concludes the proof.
\end{proof}

In light of \thref{thm:LPO<=J(id)}, a natural question is how $\jjump{\id}$
compares with~$\CNatural$ (as $\mflim \weiequiv
\parallelization{\CNatural})$. By \thref{thm:co-total_char}, as
$\CNatural$ is co-total, $\CNatural \weireducible \jjump{\id}$ would
imply that $\CNatural \weireducible \id$, which is a contradiction. On the
other hand, $\jjump{\id}\not\weireducible \CNatural$. In fact, we have a slightly stronger result (as $\CNatural \strictlyweireducible \TCN$ by \cite[Prop.\ 24]{NP18}).

\begin{proposition}\thlabel{thm:j(id)not<=TCN}
$\jjump{\id} \not\weireducible \TCN$.
\end{proposition}
\begin{proof}
We use the fact that $\jjump{\id}\weiequiv \XPi$ (\thref{thm:j(id)=x_pi}). As mentioned above, $\XPi \weiequiv \mathsf{isFinite}_\Sier$. But Neumann and Pauly proved that $\mathsf{isFinite}_\Sier\not\weireducible\TCN$ \cite[Prop.\ 24(5)]{NP18}.
\end{proof}

The following result is folklore.

\begin{lemma}\thlabel{thm:TCN_fractal}
$\TCN$ is a \emph{fractal}, i.e., for every $\sigma\in\baire$,~$\TCN$ is Weihrauch
equivalent to its restriction to $X_\sigma:=\{ p \in\Baire \st \sigma
\pprefix p \}$.
\end{lemma}

\begin{proof}
Fix $\sigma \in \baire$. To prove that $\TCN \weireducible
\TCN\restrict{X_\sigma}$, let $m:= \max \ran(\sigma)$. We can uniformly map
$p\in \Baire$ to
\[
q:= \sigma\concat \str{1,\hdots, m}\concat \str{p(0)+m,
p(1)+m,\hdots}.
\]
Clearly, $q\in X_\sigma$ and for every $n\in \TCN(q)$, $\max \{n-m, 0 \}\in \TCN(p)$.
\end{proof}

\begin{theorem}\thlabel{thm:tcn_notin_ran(j)}
There is no~$f$ such that $\TCN \weiequiv \jjump{f}$.
\end{theorem}

\begin{proof}
Observe first of all that $\TCN\strictlyweireducible \jjump{\CNatural}$.
Indeed, the reduction follows by \thref{thm:jump=max_tot}, while the fact
that the reduction is strict follows from the fact that
$\jjump{\id}\not\weireducible \TCN$ (\thref{thm:j(id)not<=TCN}) whereas
$\jjump{\id}\weireducible \jjump{\CNatural}$ (by the monotonicity of the
tot-jump).

This also implies that if $\jjump{f}\weireducible \TCN$ then
$f\strictlyweireducible \CNatural$. To conclude the proof, it is enough to
show that if $\TCN \weireducible \jjump{f}$ then $\CNatural \weireducible f$. 	

Assume now that the reduction $\TCN\weireducible \jjump{f}$ is witnessed by
the functionals $\Phi,\Psi$. Consider the input $0^\mathbb{N}$ for~$\TCN$. By continuity, there are $k\in\mathbb{N}$ and
$\sigma\in \mathbb{N}^k$ such that, for some $m\in\mathbb{N}$,
$\Psi(0^k,\sigma)(0)\downarrow = m$. 	Let $q\in \Baire$ be an input
for~$\CNatural$ such that $0^k\concat\str{m+1} \pprefix
q$.
Let also $\Phi(q) = (e,i,x)$, which is an input of $\jjump{f}$. By
definition of Weihrauch reduction, $\Phi_e(x)\in \dom(f)$ and, for every
$y\in f(\Phi_e(x))$, $\Phi_i(x,y)\downarrow \in \jjump{f}(\Phi(q))$ (as
otherwise~$\sigma$ would be the initial segment of a solution for
$\jjump{f}(\Phi(q))$, hence~$\Psi$ would not produce a valid answer for
$\TCN(q)$). This implies that
\[
\{ \Psi(q,\Phi_i(x,y)) \st y\in f(\Phi_e(x)) \} \subseteq
    \TCN(q) = \CNatural(q),
\]
where $(e,i,x)$ are uniformly computable from~$q$. Since~$\CNatural$ is a
fractal (by the same proof as \thref{thm:TCN_fractal} or see \cite[Fact 3.2(1)]{BolWei11}),
$\CNatural$ is Weihrauch equivalent to its restriction to $\{ p \in \Baire
\st 0^k\concat\str{m+1} \pprefix p \}$.

In other words, the reduction $\TCN \weireducible \jjump{f}$ yields a
reduction $\CNatural \weireducible f$, which concludes the proof.
\end{proof}

Notice that the above argument does not necessarily yield a reduction
$\TCN\weireducible f$, as $\jjump{f}$ is still allowed to go in the
``otherwise'' case when every $n>0$ is in $\ran(q)$.

\begin{remark}
Notice that~$\TCN$ is another witness for \thref{thm:j_not_onto_total_join-irred}. Indeed, it is total (trivially), a fractal (by
\thref{thm:TCN_fractal}), and not in the range of $\jjumpop$ (by
\thref{thm:tcn_notin_ran(j)}). Every fractal is join-irreducible (see \cite[Prop.\ 4.11]{BGP17}), so all that is left is to show that~$\TCN$ is not Baire co-total. If it were, then $\TCN \weireducible \TCN$ would imply that $\TCN \weireducible \CNatural$, but we have already noted that $\CNatural\strictlyweireducible \TCN$.
\end{remark}

\begin{theorem}
$\jjump{\CNatural} \weiequiv \totalization{(\CNatural \times \W)}$.
\end{theorem}

\begin{proof}
In light of \thref{thm:jump=max_tot} and of the uniform computability of
$\W$, it suffices to show that $\jjump{\CNatural}\weireducible
\totalization{(\CNatural \times \W)}$. Given the input $(e,i,p)$ for
$\jjump{\CNatural}$, let $q\in\Baire$ be such that~$q$ has infinitely many
zeroes iff
\[
(\forall n)[\; (\forall j)(\Phi_e(p)(j)\neq n+1)\rightarrow
   \Phi_i(p,n)\downarrow \;].
\]
In other words,~$q$ has infinitely many zeroes iff $\Phi_i(p,n)\downarrow$
whenever~$n$ is not enumerated by $\Phi_e(p)$.
Moreover, let~$\bar e$ be such that $\Phi_{\bar e}(p)$ works by simulating
$\Phi_e(p)$ and padding the output with zeroes (this ensures that $\Phi_{\bar
e}$ is total and $\CNatural(\Phi_e(p))=\CNatural(\Phi_{\bar e}(p))$ whenever
$\Phi_e(p)\downarrow$).

The forward functional of the reduction $\jjump{\CNatural}\weireducible
\totalization{(\CNatural \times \W)}$ is then given by the map
$(e,i,p)\mapsto (\Phi_{\bar e}(p), q)$.

Given $\pairing{n,t}\in \totalization{(\CNatural \times \W)}(\Phi_{\bar
e}(p), q)$, we uniformly compute a solution for $\jjump{\CNatural}(e,i,p)$ as
follows: We run $\Phi_i(p,n)$ until we either see that $n+1\in
\ran(\Phi_e(p))$ or we see that $t\notin \W(q)$ (both conditions are c.e.).
If this never happens, we know that~$n$ is a valid solution for
$\CNatural(\Phi_{\bar e}(p))$ and $\Phi_i(p,n)\downarrow$. Otherwise, it
means that $\jjump{\CNatural}(e,i,p)=\Baire$, hence we can just continue the
output with~$0^\mathbb{N}$.
\end{proof}

Unlike~$\CNatural$, the fact that~$\Cantor$ is computably compact implies
that $\CCantor\weiequiv \TCCantor$ (see, e.g., \cite[Prop.\
6.1]{BGCompOfChoice19}). However, a characterization similar to the one for
the tot-jump of~$\CNatural$ holds for~$\CCantor$.

\begin{theorem}
	$\jjump{\CCantor} \weiequiv \totalization{(\CCantor\times \W)}$.
\end{theorem}
\begin{proof}
	As above, it suffices to show that $\jjump{\CCantor} \weireducible \totalization{(\CCantor\times \W)}$. Let $(e,i,p)$ be an input for $\jjump{\CCantor}$. 
	Let $S\subseteq \cantor$ be a tree such that 
	\[ [S] = \{ x\in\Cantor \st (\forall n)[\;\Phi_e(p)(\coding{x[n]})\downarrow \rightarrow  \Phi_e(p)(\coding{x[n]})=1\;] \}. \]
	Observe that a tree $S$ as above can be uniformly computed from $e,p$ (as the formula defining $[S]$ is $\lightfacePi^0_1$ with parameters $e,p$) and that, if $\Phi_e(p)$ is the characteristic function of a subtree of $\cantor$, then $[S]=[\Phi_e(p)]$. Moreover, $S$ is well-defined, even if $\Phi_e(p)$ does not converge or it is not the characteristic function of a tree. To compute an input $q\in\Baire$ for $\W$, we first observe that the set 
	\begin{align*}
		A_{e,i,p}:=\{ y \in \Cantor \st  & \Phi_e(p)\uparrow  \\
			&\lor\; (\exists \tau\in 2^{<\omega})(\exists \sigma\prefix\tau)[\; \Phi_e(p)(\coding{\tau}) = 1 \land \Phi_e(p)(\coding{\sigma}) = 0 \;] \\
			& \lor\; [\; (\forall n)[\; \Phi_e(p)(\coding{y[n]})\downarrow \rightarrow  \Phi_e(p)(\coding{y[n]})=1 \;] \\ 
			&\quad\quad\quad\quad\land{} (\exists n)\; \Phi_i(p,y)(n)\uparrow \;]\}
	\end{align*}
	is defined by a $\lightfaceSigma^0_2$-formula with parameters $e,i,p$. Intuitively, the first two rows of the definition capture ``$\Phi_e(p)$ is not the characteristic function of a subtree of $\cantor$'', while the last two rows can be read as ``$y\in [\Phi_e(p)]$ and $\Phi_i(p,y)\uparrow$''. The third row could have been written as ``$(\forall n)(\Phi_e(p)(\coding{y[n]})=1)$''. This is (in general) a $\lightfacePi^0_2$ statement, but it can be equivalently rewritten in a $\lightfacePi^0_1$-way in light of the first row. 

	Since the projection of a $\lightfaceSigma^0_2$-set over a computably compact set is~$\lightfaceSigma^0_2$ (see, e.g., \cite[Lemma 3.9]{MVEffSalem}) and that checking if a subtree of $\cantor$ is ill-founded is a $\Pi^0_1$-problem, this implies that the statement 
	\[ \text{``}\Phi_e(p) \text{ is an ill-founded subtree of } \cantor \text{ and } \lnot(\exists y \in \Cantor)\; y \in A_{e,i,p}\text{.''}  \]
	is~$\Pi^0_2$. Therefore, we can uniformly compute a string $q\in\Baire$ such that $q$ has infinitely many zeroes iff the above formula holds.

	The forward functional of the reduction is the map that sends $(e,i,p)$ to $(S,q)$. The backward functional $\Psi$ is the map that works as follows: given $e,i,p$ and a solution $\pairing{z,t}$ for $\totalization{(\CCantor\times \W)}(S,q)$, $\Psi$ outputs $\Phi_i(p,z)$ as long as $t$ appears to belong to $\W(q)$. As soon as an error is found, $\Psi$ extends the partial output with $0^\mathbb{N}$. 

	It is immediate from the definition of $q$ that an error is found only if $\Phi_e(p)$ does not define an ill-founded subtree of $\cantor$ or if there is $y\in [\Phi_e(p)]$ such that $\Phi_i(p,y)\uparrow$. In this case, $\jjump{\CCantor}(e,i,p)=\Baire$, hence the computed string is a valid solution. On the other hand, if an error is never found, then $z\in [S]=[\Phi_e(p)]$ and $\Psi((e,i,p), \pairing{z,t})=\Phi_i(p,z)\in \jjump{\CCantor}(e,i,p)$, which concludes the proof.
\end{proof}

We conclude this section by showing that, as a consequence of
\thref{thm:jump=tot}, the tot-jump of each of the problems~$\mflim$, $\mflim^{[n]}$,
$\UCBaire$, and~$\CBaire$ is the respective total continuation.

\begin{theorem}
For every $n\ge 1$, $\jjump{\mflim^{[n]}} \weiequiv
\totalization{(\mflim^{[n]})}$.
\end{theorem}

\begin{proof}
Let us first prove the theorem for $n=1$. Let~$\mathcal{P}$ be the family of
all problems~$f$ such that there is a total computable~$\Phi$ such that, for
every $e,i\in\mathbb{N}$ and every $p\in\Baire$, if $\Phi_e(p)\in\dom(f)$ and
$(\forall q\in f\Phi_e(p))\; \Phi_i(p,q)\downarrow$, then
\[
\Phi(e,i,p) \in \dom(f) \quad \text{and}\quad  f\Phi(e,i,p) \subseteq
    \{ \Phi_i(p,q) \st q \in f\Phi_e(p)\}.
\]
By \thref{thm:jump=tot} (with~$\Psi$ the projection on the second coordinate), for every $f\in \mathcal{P}$, $\jjump{f}\weiequiv
\totalization{f}$, so it is enough to show that $\mflim\in\mathcal{P}$.
Let~$\Phi$ be the map that, upon input $e,i,p$, computes the sequence
$\sequence{x_n}{n\in\mathbb{N}}$ defined as follows: We read the output of
$\Phi_e(p)$ as (an initial segment of) the join of countably many strings
$\pairing{q_0,q_1,\hdots}$. Formally, if~$\sigma_n$ is the string produced by
$\Phi_e(p)$ in~$n$ steps, we define $k_n:=\max \{ j \st \pairing{j,0}<
\length{\sigma_n} \}$ and, for every $j\le k_n$, define $\tau_{n,j}(m) :=
\sigma_n(\pairing{j,m})$ for all~$m$ with $\pairing{j,m}<\length{\sigma_n}$.
With this definition, for every $j\le k_n$, $\tau_{n,j} \prefix q_j$. We can uniformly compute the string $y\in\Baire$ defined as $y_n(m) := \tau_{n,j}(m)$ for the largest $j$ such that $\pairing{j,m}<\length{\sigma_n}$, and $y_n(m):=0$ if no such $j$ exists. Observe that, if $\sequence{q_n}{n\in\mathbb{N}}$ converges, then so does $\sequence{y_n}{n\in\mathbb{N}}$ and $\lim_{n\to\infty} y_n = \lim_{n\to \infty} q_n$. We define

\[
x_n := \Phi_{i,n}(p,y_n )\concat
0^\mathbb{N},
\]
where $\Phi_{i,n}(p,t)$ denotes the output produced by $\Phi_i(p,t)$ in $n$
steps. Clearly, the map~$\Phi$ is total and computable.

Assume that $\Phi_e(p)$ produces the sequence
$\sequence{q_j}{j\in\mathbb{N}}\in\dom(\mflim)$ and that $\Phi_i(p,q )\downarrow$, where $q:=\lim_{n\to\infty} q_n$. By the continuity of~$\Phi_i$, we immediately obtain
\[
\lim_{n\to\infty} x_n = \Phi_i\left(p,\lim_{n\to\infty} y_n\right) = \Phi_i\left(p,\lim_{n\to\infty} q_n\right) = \Phi_i(p,\mflim\ \Phi_e(p)),
\]	
which shows that $\mflim\in\mathcal{P}$.

The general case follows by induction on~$n$, as the class~$\mathcal{P}$ is
closed under composition and $\mflim \circ \mflim \strongweiequiv \mflim
\compproduct \mflim$ (by \cite[Example 4.4(1)]{BolWei11} and the fact that $\mflim$ is a cylinder).
To prove that~$\mathcal{P}$ is closed under
composition, we can use $f\in \mathcal{P}$ first and $g\in \mathcal{P}$
later, to conclude that $fg\in \mathcal{P}$.

More precisely, let $f,g\in \mathcal{P}$ and let $F$ and $G$ be the total computable witnesses. Fix $e,i\in\mathbb{N}$ and
$p\in\Baire$ such that $\Phi_e(p)\in\dom(fg)$ (i.e.\ $\Phi_e(p) \in \dom(g)$ and $g(\Phi_e(p)) \subseteq \dom(f)$) and $(\forall q\in f g
\Phi_e(p))\; \Phi_i(p,q)\downarrow$. Let $\bar e,\bar \imath\in\mathbb{N}$ be
such that $\Phi_{\bar e}(\pairing{a,b})=b$ and $\Phi_{\bar
\imath}(\pairing{a,b},c) = \Phi_i(a,c)$. Observe that
\[
\{ \Phi_i(p,q) \st q \in f g\, \Phi_e(p) \}= \bigcup_{t\in g \Phi_e(p)}
 \{ \Phi_{\bar \imath}(\pairing{p,t},q) \st
 q \in f \Phi_{\bar e}(\pairing{p,t}) \}.
\]
By the choice of~$F$,
for every $t\in g\,\Phi_e(p) \subseteq \dom(f)$ (which implies that $\Phi_{\bar
e}(\pairing{p,t}) \in \dom(f)$ and $(\forall q\in f\Phi_{\bar
e}(\pairing{p,t}))\;\Phi_{\bar \imath}(\pairing{p,t},q)\downarrow$),
\begin{align*}
\emptyset \neq fF( \bar e, \bar \imath, \pairing{p,t}) &\subseteq
   \{ \Phi_{\bar \imath}( \pairing{p,t},q) \st q \in f
   \Phi_{\bar e}(\pairing{p,t}) \} \\
   &= \{ \Phi_i(p,q) \st q \in f(t) \} .
\end{align*}
In other words, letting~$j$ be such that $\Phi_j(p,t):= F(\bar e,\bar \imath,
\pairing{p,t})$, we obtain
\[
f( \{ \Phi_j(p,t) \st t \in g \Phi_e(p) \}) \subseteq
    \{ \Phi_i(p,q) \st q \in f g \Phi_e(p) \}.
\]
To conclude the proof, note that by the choice of~$G$, we have
\[
\emptyset \neq g\, G(e,j,p) \subseteq \{ \Phi_j(p,t) \st t \in g\, \Phi_e(p)
    \}.
\]
In particular, the map $(e,i,p)\mapsto G(e,j,p)$ witnesses the fact that
$f\circ g \in \mathcal{P}$.
\end{proof}

	

\begin{theorem}
$\jjump{\UCBaire} \weiequiv \totalization{\UCBaire}$ and $\jjump\CBaire
\weiequiv \TCBaire$.
\end{theorem}

\begin{proof}
We only show that $\jjump\CBaire \weiequiv \TCBaire$, as (with minor
modifications) the same proof works for~$\UCBaire$.

As in the previous proof, we show that~$\CBaire$ satisfies the assumptions
of \thref{thm:jump=tot}, namely, there are two total computable
functionals~$\Phi$ and~$\Psi$ such that, for every $e,i\in\mathbb{N}$ and
every $p\in\Baire$, if  $\Phi_e(p)\in\dom(\CBaire)$ and $(\forall x\in
[\Phi_e(p)])\; \Phi_i(p,x)\downarrow$, then $\Phi(e,i,p)\in\dom(\CBaire)$ and
$\{ \Psi((e,i,p),x) \st x\in [\Phi(e,i,p)] \} \subseteq \{ \Phi_i(p,x) \st x \in
[\Phi_e(p)]\}$. Let~$\Phi$ be the map that sends $(e,i,p)$ to (the
characteristic function of) a tree $T\subseteq \baire$ such that
\begin{align*}
   [T]= \{ \pairing{x,y}\in\Baire \st & (\forall n)(\Phi_e(p)(\coding{x[n]})\downarrow\rightarrow \Phi_e(p)(\coding{x[n]})=1) \text{ and } \\
      &  (\forall n)(\Phi_i(p,x)(n)\downarrow \rightarrow  \Phi_i(p,x)(n)= y(n)) \}. 
\end{align*}
As the set above is uniformly~$\Pi^0_1$ in $(e,i,p)$, a tree~$T$ as above can
be uniformly computed from $(e,i,p)$. Let $\Psi:= ((e,i,p),
\pairing{x,y})\mapsto y$.

To conclude the proof, notice that if $\Phi_e(p)\in\dom(\CBaire)$ and
$\Phi_i(p,x)\downarrow$ for every $x\in [\Phi_e(p)]$, then~$T$ is ill-founded and, for every $\pairing{x,y}\in [T]$, $\Psi((e,i,p), \pairing{x,y}) = y \in \{
\Phi_i(p,q) \st q \in [\Phi_e(p)] \}$.
\end{proof}

\section{Remarks on abstract jump operators}\label{sec:conclusions}

In this paper, we introduce and study a natural jump operator on the Weihrauch lattice, a natural partial order. The remarks in this section address a much more abstract question: under what conditions do arbitrary partial orders admit a jump operator in the sense of \thref{def:jump}? We show that, without additional structure, admitting a jump is not a first-order property.

We mentioned in the introduction that the existence of an abstract jump
operator is easy to see in any upper semilattice without maximum (using the
Axiom of Choice). This result can be extended to countable upper directed
partial orders (i.e., partial orders such that any finite number of elements
have a common upper bound) without maximum. Indeed, any such partial order
$(P,\le)$ admits a strictly increasing cofinal chain
$\sequence{q_n}{n\in\mathbb{N}}$. This can be easily shown by letting
$\sequence{p_n}{n\in\mathbb{N}}$ be an enumeration of~$P$ and defining
$q_0:=p_0$ and $q_{n+1}:= p_m$ where~$m$ is least such that $q_n<p_m$. The
existence of a jump operator follows from the fact that every partial order
with a strictly increasing cofinal chain admits a jump operator (it is enough
to map every element of the poset to the first element in the sequence that
is strictly above it).

We mention that the same strategy cannot be used to obtain the existence of a
jump operator on the Weihrauch degrees. Indeed, the third and fifth author will show in an upcoming paper that no chain in the Weihrauch degrees can be cofinal.

\begin{remark}
There is an upper directed partial order with no maximum and size $\aleph_1$
that does not admit a jump operator, which implies that the above observation
cannot be generalized to larger partial orders. To show this, let $(Q,
\le_Q)$ be~$\omega_1$ ordered as an antichain (each element is only
comparable with itself). Let also $(R,\le_R)$ be the partial order of all
non-empty finite subsets of~$\omega_1$ ordered by inclusion. Let~$P$ be the
union of~$Q$ and~$R$, where~$\le_P$ is defined as the transitive closure of
\[
\le_Q \cup \le_R \cup~ \{ (\alpha, \{\gamma \}) \st \alpha \le \gamma
   \}.
\]

Assume towards a contradiction that~$P$ admits a jump operator~$j$. Observe
now that, for a fixed~$\alpha$, $j(\alpha)$ is a non-empty finite set of
ordinals such that at least one is $\ge \alpha$. Let $\beta \in j(\alpha)$ be
such that $\alpha \le \beta$. In particular, for every $\gamma \ge \alpha$,
we have $\alpha <_P \{ \gamma \}$ and hence $j(\alpha) \le_P j(\{ \gamma
\})$, i.e., $\beta \in j(\alpha) \subseteq j(\{\gamma\})$.

Define an increasing sequence $\sequence{\alpha_n}{n\in\mathbb{N}}$ of
countable ordinals as follows: $\alpha_0:=0$ and, for every~$n$,
$\alpha_{n+1}:=\max j(\alpha_n) +1$. For every~$n$, let $\beta_n\in
j(\alpha_n)$ be such that $\alpha_n \le \beta_n$. In particular, we obtain
\[
\alpha_0 \le \beta_0 < \alpha_1 \le \beta_1 \hdots,
\]
which implies that all the~$\beta_n$ are distinct.

By the above observation, for every $\gamma\ge \sup_{n\in\mathbb{N}}
\alpha_n$ and every~$n$, $\beta_n \in j(\{\gamma\})$, hence $j(\{\gamma\})$ is infinite, which is a
contradiction with $j(\{ \gamma \})\in P$.
\end{remark}

Finally, we consider the case of arbitrary countable partial orders (without a maximum). We show that the existence of a jump operator is as complicated as its naive definition suggests: it is $\Sigma^1_1$-complete.

\begin{theorem}
There is a computable map $F\function{\mathrm{LO}}{\mathrm{PO}_0}$ from the
family of countable linear orders to the family of countable partial orders
without maximum (where we represent a linear/partial order using its
characteristic function) such that
\[
L \text{ is ill-founded } \iff F(L) \text{ admits a jump operator.}
\]		
\end{theorem}

\begin{proof}
To show this, let $(X, \le_X)$ be the partial order defined as $X := 2 \times
\omega$ and $(i,n) \le_X (j,m)$ iff $i=j$ and $n\le m$. Intuitively,~$X$
consists of two incomparable copies of~$\omega$. We define $F(L):= 1+\sum_{x
\in L^*} X$, where~$L^*$ is~$L$ with the order reversed. We order $F(L)$ as expected: in particular, every element of a given term is less than every element of the later terms. For the sake of
readability, we write $X_x:= \{ (i,n)_x \st i\in 2 \text{ and } n \in\omega
\}$ for the~$x$-th copy of~$X$ in $F(L)$. Let also~$\bot$ be the minimum
of~$F(L)$.

Assume that~$L$ is ill-founded and let $\sequence{x_k}{k\in\mathbb{N}}$ be an
infinite descending sequence (i.e., an infinite ascending sequence in~$L^*$).
We can define a jump function~$j$ on~$F(L)$ as follows:
\begin{itemize}
\item
$j(\bot) :=(0,0)_{x_0}$;
\item
for every $(i,n)_x$ such that $x <_{L^*} x_0$, let
$j((i,n)_x):=(0,0)_{x_0}$;
\item
for every $(i,n)_x$ such that $x_k \leq_{L^*} x <_{L^*} x_{k+1}$, let
$j((i,n)_x):=(0,0)_{x_{k+1}}$;
\item
for every $(i,n)_x$ such that $x_k <_{L^*} x$ for every~$k$, let
$j((i,n)_x):=(i,n+1)_x$.
\end{itemize}
It is immediate from the definition that~$j$ is strictly increasing. Proving
that~$j$ is weakly monotone is also easy. 	

Conversely, if $F(L)$ admits a jump operator~$j$, then we can define an
ascending sequence in~$L^*$ (i.e., a descending sequence in~$L$ witnessing
the fact that~$L$ is ill-founded) as follows: We let~$x_0$ be such that
$j(\bot) = (i_0,n_0)_{x_0}$ for some $i_0,n_0$. To define~$x_1$, observe
that, by the weak monotonicity of~$j$, we have $j((1-i_0,n_0)_{x_0}) \ge_{F(L)}
j(\bot) = (i_0,n_0)_{x_0}$. This, in combination with $j((1-i_0,n_0)_{x_0})
>_{F(L)} (1-i_0,n_0)_{x_0}$, implies that $j((1-i_0,n_0)_{x_0}) =
(i_1,n_1)_{x_1}$ for some $i_1,n_1$ with $x_1 >_{L^*} x_0$. We can
iterate this argument to obtain the desired strictly
increasing sequence in~$L^*$.		
\end{proof}

This shows that the set of partial orders without maximum admitting a jump
operator is a non-Borel $\Sigma^1_1$-subset of the set of countable partial
orders. In particular, the existence of a jump operator for countable partial
orders without maximum cannot be characterized by an arithmetic formula. It
would be interesting to obtain a similar result for non-countable
partial orders. This would require a detour into the realm of generalized
descriptive set theory, and it is possible that additional set-theoretic axioms
(for example, on the size of the continuum) would be needed.

\section{Open Questions}\label{sec:openQ}

We mentioned in Section~\ref{sec:tot-jump} that the tot-jump of~$f$ can be
defined via a $\Delta^1_2$-formula using~$f$ as a parameter. Moreover, we
showed that no jump operator on computational problems can be defined using a
$\Sigma^{1,f}_1$-formula (\thref{rem:no_sigma11_def}). This leaves a gap, and
therefore it is natural to ask the following question:

\begin{open}
Is there a $\Pi^{1,f}_1$-definition for the tot-jump? More generally, is
it possible to define a jump operator on the Weihrauch degrees using a
$\Pi^{1,f}_1$-formula?
\end{open}

Despite our efforts, we could not obtain a satisfactory characterization for $\ran(\jjumpop)$. 
A better
characterization is especially desirable in light of
\thref{thm:jump_endomorphism}, as that would give us a description of a
sublattice of the Weihrauch degrees which is isomorphic to the full
structure.

\begin{open}
Find a ``natural'' characterization for $\ran(\jjumpop)$. Is $\jjumpop$ definable in the Weihrauch degrees?
\end{open}

While the ``natural'' condition is of course vague and informal, a
satisfactory answer would allow us to promptly tell whether a given~$f$ is in
the range of the tot-jump. To this end, a powerful result is provided by
\thref{thm:dis_x_g}, and especially by \thref{thm:product_dis}. As proved,
closure under product with~$\DIS$ is a sufficient condition for Baire co-totality,
and we showed that a Baire co-total problem~$f$ can be Weihrauch-reducible to
$\jjump{g}$ only in the trivial case $f\weireducible g$. This raises the
following question:

\begin{open}
Does closure under product with~$\DIS$ characterize Baire co-totality?
\end{open}

We also showed that the range of~$\jjumpop$ is a (proper) subset of the
join-irreducible degrees (\thref{thm:total_and_join_irreducible} and \thref{thm:j_not_onto_total_join-irred}). This
allowed us to show that the tot-jump of~$f$ and~$g$ distributes over the join
of~$f$ and~$g$ only when~$f$ and~$g$ are comparable. On the other hand, we do
not know whether the same holds for the meet: While \thref{thm:tj-meet} shows
that, in general, the jump does not distribute over the meet, we do not know
whether this is always the case when~$f$ and~$g$ are not comparable.

\begin{open}
Are there~$f$ and~$g$ such that $f \weiincomparable g$ but $\jjump{f} \sqcap
\jjump{g} \weiequiv \jjump{f \sqcap g}$?
\end{open}

%
%

\bibliographystyle{mbibstyle}
\bibliography{bibliography}

\end{document}